\documentclass[12pt,a4paper]{amsart}

\makeatletter

\usepackage[english]{babel}
\usepackage[utf8]{inputenc}
\usepackage{csquotes}
\usepackage[T1]{fontenc}
\usepackage{latexsym}
\usepackage[leqno]{amsmath}
\usepackage{amssymb,amsthm,amsfonts}
\usepackage{mathtools,thmtools}
\usepackage{dsfont}
\usepackage{stmaryrd}
\usepackage{cancel}
\usepackage{xcolor}

\usepackage{todonotes}

\usepackage{geometry}
\usepackage{nicefrac}

\usepackage{enumerate}
\usepackage[shortlabels]{enumitem}
\setlist[enumerate]{label=(\alph*),font=\normalshape}
\setlist[itemize]{font=\normalshape}
\let\originalitem\item
\renewcommand{\item}[1][]{%
	\if\relax\detokenize{#1}\relax%
		\originalitem%
	\else%
		\originalitem[#1]%
		\phantomsection
		\def\@currentlabel{#1}
	\fi%
}

\usepackage{hyperref}

\usepackage[noabbrev,capitalize]{cleveref}

\usepackage{refcheck}
\newcommand{\refcheckize}[1]{%
  \expandafter\let\csname @@\string#1\endcsname#1%
  \expandafter\DeclareRobustCommand\csname relax\string#1\endcsname[1]{%
    \csname @@\string#1\endcsname{##1}\wrtusdrf{##1}}%
  \expandafter\let\expandafter#1\csname relax\string#1\endcsname
}
\refcheckize{\cref}
\refcheckize{\Cref}

\usepackage[style=numeric-comp]{biblatex}

\let\originalleft\left
\let\originalright\right
\renewcommand{\left}{\mathopen{}\mathclose\bgroup\originalleft}
\renewcommand{\right}{\aftergroup\egroup\originalright}

\definecolor{kitgreen}{RGB}{0,150,130}
\definecolor{kitblue}{RGB}{70,100,170}
\definecolor{kitmaygreen}{RGB}{140,182,60}
\definecolor{kityellow}{RGB}{252,229,0}
\definecolor{kitorange}{RGB}{223,155,27}
\definecolor{kitbrown}{RGB}{167,130,46}
\definecolor{kitred}{RGB}{162,34,35}
\definecolor{kitpurple}{RGB}{163,16,124}
\definecolor{kitcyanblue}{RGB}{35,161,224}

\allowdisplaybreaks

\theoremstyle{plain}
\newtheorem{theorem}{Theorem}[section]
\newtheorem{definition}[theorem]{Definition}
\newtheorem{proposition}[theorem]{Proposition} 
\newtheorem{lemma}[theorem]{Lemma} 
\newtheorem{corollary}[theorem]{Corollary}
\theoremstyle{definition}
\newtheorem{remark}[theorem]{Remark}

\newcommand{\C}{\mathbb{C}} 
\newcommand{\R}{\mathbb{R}} 
 
\newcommand{\Z}{\mathbb{Z}} 

\newcommand{\Zeven}{\mathbb{Z}_\mathrm{even}}
\newcommand{\N}{\mathbb{N}} 
\newcommand{\Nodd}{\mathbb{N}_\mathrm{odd}}

\newcommand{\T}{\mathbb{T}} 
\newcommand{\F}{\mathcal{F}}

\newcommand{\der}[2][]{\@ifnextchar\der{\,#1\mathrm{d}#2\!}{\,#1\mathrm{d}#2}}
\newcommand{\Der}{\mathrm{d}}

\DeclareMathOperator*{\supp}{supp}

\DeclareMathOperator*{\esssup}{ess~sup}

\renewcommand{\Re}{\operatorname{Re}}

\newcommand{\calB}{\mathcal{B}}

\newcommand{\calE}{\mathcal{E}}
\newcommand{\calF}{\mathcal{F}}
\newcommand{\calG}{\mathcal{G}}

\newcommand{\calM}{\mathcal{M}}
\newcommand{\calN}{\mathcal{N}}

\newcommand{\bfD}{\mathbf{D}}
\newcommand{\bfE}{\mathbf{E}}
\newcommand{\bfB}{\mathbf{B}}
\newcommand{\bfH}{\mathbf{H}}
\newcommand{\bfx}{\mathbf{x}}
\newcommand{\bftx}{\tilde{\mathbf{x}}}

\newcommand{\bfP}{\mathbf{P}}

\newcommand{\bfu}{\mathbf{u}}
\newcommand{\bfv}{\mathbf{v}}
\newcommand{\bfw}{\mathbf{w}}

\newcommand{\SO}{\mathrm{SO}}
\newcommand{\landauO}{\mathcal{O}}
\newcommand{\landauo}{\hbox{o}}

\newcommand{\ee}{\mathrm{e}}
\newcommand{\ii}{\mathrm{i}}
\let\eps\varepsilon

\let\wto\rightharpoonup
\let\embeds\hookrightarrow

\newcommand{\set}[1]{\left\{ #1 \right\}}
\newcommand{\abs}[1]{\left\lvert #1 \right\rvert}

\newcommand{\norm}[1]{\left\lVert #1 \right\rVert}
\newcommand{\Norm}[1]{\bigl\lVert #1 \bigr\rVert}
\newcommand{\nnorm}[1]{\left\vert\kern-0.25ex\left\vert\kern-0.25ex\left\vert #1 \right\vert\kern-0.25ex\right\vert\kern-0.25ex\right\vert}

\newcommand{\ceil}[1]{\left\lceil #1 \right\rceil}

\newcommand{\impvar}{\,\cdot\,}
\newcommand{\ip}[2]{\left\langle #1 , #2 \right\rangle}
\newcommand{\Ip}[2]{\bigl\langle #1 , #2 \bigr\rangle}
\newcommand{\iip}[2]{\langle\!\langle #1 , #2 \rangle\!\rangle}

\newcommand{\pdv}[3][]{%
	\if\relax\detokenize{#1}\relax%
	\frac{\partial #2}{\partial #3}%
	\else%
	\frac{\partial^{#1} #2}{\partial {#3}^{#1}}%
	\fi%
}
\newcommand{\dv}[3][]{%
	\if\relax\detokenize{#1}\relax%
	\frac{\mathrm{d} #2}{\mathrm{d} #3}%
	\else%
	\frac{\mathrm{d}^{#1} #2}{\mathrm{d} {#3}^{#1}}%
	\fi%
}

\newlength{\negph@wd}
\newcommand{\negphantom}[1]{%
	\ifmmode
		\mathpalette\negph@math{#1}%
	\else
		\negph@do{#1}%
	\fi
}
\newcommand{\negph@math}[2]{\negph@do{$\m@th#1#2$}}
\newcommand{\negph@do}[1]{%
	\settowidth{\negph@wd}{#1}%
	\hspace*{-\negph@wd}%
}

\newcommand{\clonelabel}[2]{\@bsphack
	\expandafter\ifx\csname r@#2\endcsname\relax
	\else\protected@write\@auxout{}{\string\newlabel{#1}%
		{\csname r@#2\endcsname}}%
	\fi
	\expandafter\ifx\csname r@#2@cref\endcsname\relax
	\else\protected@write\@auxout{}{\string\newlabel{#1@cref}%
		{\csname r@#2@cref\endcsname}}%
	\fi
	\@esphack}

\DeclareMathAlphabet{\othermathbb}{U}{bbold}{m}{n}
\newcommand{\bbone}{\othermathbb{1}}

\newcommand{\begin@color}[1]{\begingroup\color{#1}} 
\newcommand{\beginold}{\begin@color{red}}
\newcommand{\beginnew}{\begin@color{blue}}

\let\end@color\endgroup
\let\endold\end@color
\let\endnew\end@color

\newcommand{\derf}[2][]{%
	\if\relax\expandafter\detokenize{#1}\relax%
		\frac{\mathrm{d} #2}{\sqrt{2 \pi}}%
	\else%
		\frac{\mathrm{d} #2}{(2 \pi)^{\nicefrac{#1}{2}}}%
	\fi%
}

\newcommand{\singularK}{\mathfrak{S}}
\newcommand{\regularK}{\mathfrak{R}}
\newcommand{\loc}{\mathrm{loc}}
\newcommand{\per}{\mathrm{per}}
\newcommand{\rad}{\mathrm{rad}}
\newcommand{\gs}{\mathrm{gs}}
\newcommand{\mopa}{\mathrm{mp}}
\newcommand{\fracDT}[1]{\abs{\partial_t}^{#1}}

\newcommand{\fc}{J}
\newcommand{\fcd}{H}
\newcommand{\fcr}{{\tilde J}}
\newcommand{\fcrd}{{\tilde H}}

\newcommand{\projT}{P_\T}
\newcommand{\projR}{P_\regularK}
\newcommand{\projS}{P_\singularK}

\newcommand{\islab}{.1}
\newcommand{\icyl}{.2}
\newcommand{\inoni}{.i}
\newcommand{\inonii}{.ii}

\norefnames

\makeatother

\bibliography{bibliography}

\begin{document}
	\title[Travelling breather solutions for cubic nonlinear Maxwell equations]{Travelling breather solutions in waveguides for cubic nonlinear Maxwell equations with retarded material laws}

	\author{Sebastian Ohrem}
	\address{Institute for Analysis, Karlsruhe Institute of Technology (KIT), D-76128 Karlsruhe, Germany}\email{sebastian.ohrem@kit.edu}
	\author{Wolfgang Reichel}
	\address{Institute for Analysis, Karlsruhe Institute of Technology (KIT), D-76128 Karlsruhe, Germany}\email{wolfgang.reichel@kit.edu}

	\date{\today}
	\subjclass[2000]{Primary: 35Q61, 49J10; Secondary: 35C07, 78A50}
	\keywords{Maxwell equations, waveguides, nonlinear material law, retardation, polychromatic breather solutions, variational method}


\begin{abstract}
 For Maxwell's equations with nonlinear polarization we prove the existence of time-periodic breather solutions travelling along slab or cylindrical waveguides. The solutions are TE-modes which are localized in space directions orthogonal to the direction of propagation. We assume a magnetically inactive and electrically nonlinear material law with a linear $\chi^{(1)}$- and a cubic $\chi^{(3)}$-contribution to the polarization. The $\chi^{(1)}$-contribution may be retarded in time or instantaneous whereas the $\chi^{(3)}$-contribution is always assumed to be retarded in time. We consider two different cubic nonlinearities which provide a variational structure under suitable assumptions on the retardation kernels. By choosing a sufficiently small propagation speed along the waveguide the second order formulation of the Maxwell system becomes essentially elliptic for the $\bfE$-field so that solutions can be constructed by the mountain pass theorem. The compactness issues arising in the variational method are overcome by either the cylindrical geometry itself or by extra assumptions on the linear and nonlinear parts of the polarization in case of the slab geometry. Our approach to breather solutions in the presence of time-retardation is systematic in the sense that we look for general conditions on the Fourier-coefficients in time of the retardation kernels. Our main existence result is complemented by concrete examples of coefficient functions and retardation kernels.    
\end{abstract}

	\maketitle


\section{Introduction}

We show existence and regularity of spatially localized, real-valued and time-periodic solutions (called breathers) to Maxwell's equations
\begin{align} \label{eq:maxwell}
    \begin{aligned}
	&\nabla \cdot \bfD = 0, 
	&&\nabla \times \bfE = - \bfB_t, 
	\\
	&\nabla \cdot \bfB = 0, 
	&&\nabla \times \bfH = \bfD_t,
	\end{aligned}
\end{align}
without charges and currents. \eqref{eq:maxwell} is posed on all of $\R^3$ with an underlying material that is either a slab waveguide or a cylindrically symmetric waveguide. We look for solutions that are travelling parallel to the direction of the waveguide, and which are transverse-electric, i.e. the electric field $\bfE$ is orthogonal to the direction of travel.
We assume that the material satisfies the constitutive relations
\begin{align}\label{eq:material}
	\bfB = \mu_0 \bfH, 
	\qquad
	\bf D = \epsilon_0 \bfE + \bfP(\bfE)
\end{align}
where $\mu_0, \eps_0 \in (0, \infty)$ are the vacuum permeability and permittivity, respectively. This means that the material is magnetically inactive. However, the displacement field $\bfD$ depends nonlinearly on the electric field $\bfE$ through the polarization field $\bfP(\bfE)$, which is modeled as a sum of a linear and a cubic function of $\bfE$. Both parts are local in space but nonlocal in time (cf. \cite{agrawal} for a physical motivation) and are given by
\begin{multline}\label{eq:polarization}
		\bfP(\bfE)(x,t) =\epsilon_0 \int_0^\infty \chi^{(1)}(\bfx, \tau)[\bfE(\bfx, t - \tau)] \der \tau
		\\ +\epsilon_0 \int_0^\infty \int_0^\infty \int_0^\infty \chi^{(3)}(\bfx, \tau_1, \tau_2, \tau_3)[\bfE(\bfx, t - \tau_1),\bfE(\bfx, t - \tau_2),\bfE(\bfx, t - \tau_3)] \der{\tau_1}\der{\tau_2}\der{\tau_3}.
\end{multline} 
Here $\bfx = (x,y,z)$ denotes the spatial variable, the susceptibility tensor $\chi^{(1)}(\bfx, \tau) \colon \R^3 \to \R^3$ is linear and $\chi^{(3)}(\bfx, \tau_1, \tau_2, \tau_3) \colon \R^3 \times \R^3 \times \R^3 \to \R^3$ is trilinear. 

By taking the curl of Faraday's law $\nabla \times \bfE = -\bfB_t$, we obtain from \eqref{eq:maxwell},\eqref{eq:material} the second order form of Maxwell's equations
\begin{align}\label{eq:2nd_order_maxwell}
	\nabla \times \nabla \times \bfE + \epsilon_0 \mu_0 \partial_t^2 \bfE + \mu_0 \partial_t^2 \bfP(\bfE) = 0.
\end{align}
While \eqref{eq:2nd_order_maxwell} is an equation only for $\bfE$, the other electromagnetic fields can be recovered if \eqref{eq:2nd_order_maxwell} holds: $\bfB$ is obtained from $\nabla \times \bfE = - \bfB_t$ by time-integration, and $\bfH, \bfD$ are then determined by the material laws \eqref{eq:material}. 
Next, $\bfB$ is divergence-free if it is divergence-free at time $0$ since $\partial_t \nabla \cdot \bfB = -\nabla \cdot(\nabla \times \bfE) = 0$. 
Lastly, $\bfD = \eps_0 \bfE + \bfP(\bfE)$ will be divergence-free because of the choices of $\bfE, \bfP$ made later.

We assume that the material is either a slab waveguide or a cylindrical waveguide. In the first case, the susceptibility tensors $\chi^{(j)}$ remain constant as $\bfx$ moves parallel to the slab. Assuming that the slab is given by $\set{x = 0}$, this means that
\begin{align}\label{eq:material:slab}\refstepcounter{equation}\tag{\theequation\islab}
	\chi^{(1)}(\bfx, \tau) = \chi^{(1)}(x, \tau), 
	\qquad
	\chi^{(3)}(\bfx, \tau_1, \tau_2, \tau_3) = \chi^{(3)}(x, \tau_1, \tau_2, \tau_3).
\end{align}
If instead the underlying material has a cylindrical waveguide geometry, we assume that the susceptibility tensors $\chi^{(j)}$ depend only on 
the distance from $\bfx$ to the cylinder core which we assume to be given by $\set{x = y = 0}$, so that
\begin{align}\label{eq:material:cyl}\tag{\theequation\icyl}
	\chi^{(1)}(\bfx, \tau) = \chi^{(1)}(r, \tau), 
	\qquad
	\chi^{(3)}(\bfx, \tau_1, \tau_2, \tau_3) = \chi^{(3)}(r, \tau_1, \tau_2, \tau_3).
\end{align}
where $r = \sqrt{x^2 + y^2}$.

With $I$ we denote the $3\times3$ identity matrix. We assume that the materials are isotropic, i.e.
\begin{align*}
	\chi^{(1)}(\bfx, \tau)[O \bfv] 
	= O\chi^{(1)}(\bfx, \tau)[\bfv], 
	\quad
	\chi^{(3)}(\bfx, \tau_1, \tau_2, \tau_3)[O \bfu, O \bfv, O \bfw]	
	= O \chi^{(3)}(\bfx, \tau_1, \tau_2, \tau_3)[\bfu, \bfv, \bfw]
\end{align*}
holds for $O \in \SO(3)$. This means that $\chi^{(1)}(\bfx, \tau) \in \R I$. For $\chi^{(3)}$ a variety of isotropic scenarios are possible, but in this paper we only consider two kinds of  nonlinear material responses: either
\begin{align}\label{eq:material:noni}\refstepcounter{equation}\tag{\theequation\inoni}
	&\chi^{(3)}(\bfx, \tau_1, \tau_2, \tau_3)[\bfu, \bfv, \bfw]
	= h(\bfx) \nu(\tau_1) \delta(\tau_2 - \tau_1) \delta(\tau_3 - \tau_1) \ip{\bfu}{\bfv} \bfw
\intertext{or}\label{eq:material:nonii}\tag{\theequation\inonii}
	&\chi^{(3)}(\bfx, \tau_1, \tau_2, \tau_3)[\bfu, \bfv, \bfw]
	= h(\bfx) \nu(\tau_1) \nu(\tau_2) \nu(\tau_3) \ip{\bfu}{\bfv} \bfw
\end{align}
where $\delta$ denotes the dirac measure at $0$, $\ip{\impvar}{\impvar}$ is the standard inner product on $\R^3$, and $h, \nu$ are given real-valued functions.

For these material laws, we will see that \eqref{eq:2nd_order_maxwell} can be viewed as a variational problem, and we will use a simple mountain-pass method in order to construct breather solutions to \eqref{eq:maxwell},\eqref{eq:material}. We deal with the kernel of the curl operator in \eqref{eq:2nd_order_maxwell} by looking for breather solutions in special divergence-free ansatz spaces that we discuss next.
For the slab geometry \eqref{eq:material:slab} we make the TE-polarized traveling wave ansatz
\begin{align}\label{eq:ansatz:slab}\refstepcounter{equation}\tag{\theequation\islab}
	\bfE(\bfx, t) = w(x, t - \tfrac{1}{c} z) \cdot \begin{pmatrix} 0 \\ 1 \\ 0 \end{pmatrix}
\end{align}
where $w \colon \R \times \R \to \R$ is periodic in the second variable, which we again denote by $t$. For the cylindrical geometry \eqref{eq:material:cyl} we instead consider the circular TE-polarized traveling wave ansatz
\begin{align}\label{eq:ansatz:cyl}\tag{\theequation\icyl}
	\bfE(\bfx, t) = w(r, t - \tfrac{1}{c} z) \cdot \begin{pmatrix} -\nicefrac yr \\ \nicefrac xr \\ 0 \end{pmatrix}
\end{align}
with $r = \sqrt{x^2 + y^2}$ and $w \colon (0, \infty) \times \R \to \R$ being periodic in $t$. Both ansatzes for the electric field are divergence-free, so that $\nabla\times\nabla\times \bfE = -\Delta \bfE$ holds, and are of a simple, essentially two-dimensional form, which greatly simplifies the discussion. More specifically, for the slab ansatz \eqref{eq:ansatz:slab} problem \eqref{eq:2nd_order_maxwell} reduces to
\begin{align}\refstepcounter{equation}\label{eq:profile_problem:slab}\tag{\theequation\islab}
	(- \partial_x^2 - \tfrac{1}{c^2} \partial_t^2) w + \epsilon_0 \mu_0 \partial_t^2 w + \mu_0 \partial_t^2 P(w)
	= 0
\end{align}
and for the cylindrical ansatz \eqref{eq:ansatz:cyl} to 
\begin{align}\label{eq:profile_problem:cyl}\tag{\theequation\icyl}
	(- \partial_r^2 - \tfrac{1}{r} \partial_r + \tfrac{1}{r^2}  - \tfrac{1}{c^2} \partial_t^2) w + \epsilon_0 \mu_0 \partial_t^2 w + \mu_0 \partial_t^2 P(w)
	= 0.
\end{align}
Here, depending on the choice of the nonlinear susceptibility tensor, the scalar polarization $P$ is given either by
\begin{align}\refstepcounter{equation}\label{eq:scalar_polarization:noni}\tag{\theequation\inoni}
	P(w)(\bfx, t) &= \epsilon_0\int_0^\infty \chi^{(1)}(\bfx, \tau) w(\bfx, t - \tau) \der \tau
	+ \epsilon_0 h(\bfx) \int_0^\infty w(\bfx, t - \tau)^3 \nu(\tau)\der{\tau}
\intertext{or by}\label{eq:scalar_polarization:nonii}\tag{\theequation\inonii}
	P(w)(\bfx, t) &= \epsilon_0\int_0^\infty \chi^{(1)}(\bfx, \tau) w(\bfx, t - \tau) \der \tau
	+ \epsilon_0 h(\bfx) \left(\int_0^\infty w(\bfx, t - \tau) \nu(\tau)\der{\tau}\right)^3
\end{align}
for susceptibilities \eqref{eq:material:noni} and \eqref{eq:material:nonii}, respectively. 
The simple form of the nonlinearity in $P(w)$, especially that the variables $\bfx$ and $\tau$ appear separated, are needed in order to obtain a variational problem. 
We denote by $\ast$ the convolution in time and normalize the speed of light to $c_0^2 = (\epsilon_0 \mu_0)^{-1} = 1$. Then problem \eqref{eq:profile_problem:slab} with polarization \eqref{eq:scalar_polarization:noni}, which we discuss as an example, becomes 
\begin{align*}
	\left(-\partial_x^2 + \left( 1 - \tfrac{1}{c^2} + \chi^{(1)} \ast \right) \partial_t^2\right) w + h(\bfx) (\nu \ast \partial_t^2) w^3 
	= 0.
\end{align*}
Inverting the convolution operator $\nu\ast\partial_t^2$ formally\footnote{rigorous considerations are given later}, we then obtain
\begin{align} \label{eq:example}
	\left(\nu\ast\partial_t^2 \right)^{-1}\left(-\partial_x^2 + \left( 1 - \tfrac{1}{c^2} + \chi^{(1)} \ast\right) \partial_t^2\right) w + h(\bfx) w^3 
	= 0.
\end{align}
Given our assumptions, we can ensure that the linear operator is symmetric when restricted to suitable spaces of time-periodic functions. Hence solutions formally are critical points of the functional
\begin{align} \label{eq:example_functional}
	\fc(w) = \int \left(\tfrac12 w \cdot \left(\nu\ast\partial_t^2\right)^{-1}\left(-\partial_x^2 + \left( 1 - \tfrac{1}{c^2} + \chi^{(1)} \ast \right) \partial_t^2\right) w + \tfrac1{4} h(\bfx) w^4\right) \der(\bfx, t).
\end{align}
Using the mountain-pass method, we will find critical points and then show that they are breather solutions to Maxwell's equations \eqref{eq:maxwell},\eqref{eq:material}.

In the literature, there are several papers treating existence of breather solutions of \eqref{eq:2nd_order_maxwell}. Many authors have considered monochromatic solutions, i.e. solutions of the form $\bfE(\bfx, t) = \calE(\bfx) \ee^{\ii \omega t} + c.c.$, where the complex conjugate is necessary in order to make the $\bfE$-field real valued. This is a viable approach if one ignores the higher-order harmonics $\ee^{\pm 3 \ii \omega t}$ coming from the nonlinear part of the polarization, or if one considers a nonlinear susceptibility tensor given by
\begin{align} \label{eq:special}
	\chi^{(3)}(\bfx, \tau_1, \tau_2, \tau_3)[\bfu, \bfv, \bfw] 
	= \tfrac{1}{2T} h(\bfx) \bbone_{[0, T]}(\tau_1) \delta(\tau_2 - \tau_1) \delta(\tau_3) \ip{\bfu}{\bfv} \bfw,
\end{align}
where $T = \frac{2 \pi}{\omega}$ is the period of the breather $\bfE$. Both approaches lead to the nonlinear curl-curl problem
\begin{align}\label{eq:semilinear_curl_curl}
	\nabla \times \nabla \times \calE 
	- \omega^2 g(\bfx) \calE - \omega^2 h(\bfx) \abs{\calE}^2 \calE
	= 0
\end{align}
which is variational provided $g(\bfx) = \int_0^\infty \chi^{(1)}(\bfx, \tau) \ee^{\ii \omega \tau} \der \tau$ is real valued. Instead of the cubic nonlinearity $h(\bfx) \abs{\calE}^2 \calE$, saturated nonlinearities $h(\bfx, \abs{\calE}^2) \calE$, which grow linearly as $\abs{\calE} \to \infty$, are also of interest and were first investigated by Stuart et al. \cite{McLeod92,Stuart04,Stuart91,StuartZhou03,StuartZhou96,StuartZhou10,StuartZhou05,StuartZhou01}. In these papers divergence-free, traveling, TE- or TM-polarized ansatz functions similar to \eqref{eq:ansatz:cyl} were used to reduce the Maxwell problem to an elliptic one-dimensional problem and to solve it via variational methods. 
An extension of Stuart's approach to more general wave-guide profiles was given in \cite{Mederski_Reichel}. Standing monochromatic breathers composed of axisymmetric divergence free ansatz functions were considered in \cite{azzolini_et_al, bartsch_et_al, benci_fortunato}. 
The next step forward to overcome special divergence free ansatz functions was accomplished by Mederski et al. \cite{MederskiENZ,MederskiSchino22,MederskiSchinoSzulkin20, BartschMederskiSurvey} who considered the full curl-curl problem \eqref{eq:semilinear_curl_curl}, also for more general nonlinearities $\partial_\calE h(\bfx, \calE)$. The difficulties arising from the infinite-dimensional kernel of $\nabla \times$ where overcome by a Helmholtz decomposition and suitable profile decompositions for Palais-Smale sequences. 
Alternative approaches used limiting absorption principles \cite{mandel_lap}, dual variational approaches \cite{mandel_dual, mandel_dual_nonlocal}, approximations near gap edges of photonic crystals \cite{dohnal1}, and monochromatic time-decaying solutions at interfaces of metals and dielectrics \cite{dohnal2, dohnal2_corr, Dohnal_He}. In the latter series of papers, also time-periodic solutions can be found if one additionally assumes $\mathcal{PT}$-symmetry of the materials. 

If one does not want to rely on very specific retardation kernels as in \eqref{eq:special} or if one wants to take higher harmonics into account then one is naturally led to polychromatic breather solutions. In the context of instantaneous material laws they have recently received increasing attention. As a model problem consider
\begin{align*}
	\nabla \times \nabla \times \calE 
	+ g(\bfx) \partial_t^2 \calE 
	+ h(\bfx) \partial_t^2 (\abs{\calE}^2 \calE)
	= 0
\end{align*}
For this problem, rigorous existence result for travelling breathers in the slab geometry \eqref{eq:material:slab} where either $g$ or $h$ contains delta distributions are given in \cite{kohler_reichel} by variational methods and in \cite{bruell_idzik_reichel} via bifurcation theory. 
Even earlier in \cite{PelSimWein} the authors used a combination of local bifurcation theory and continuation methods in a partly analytical and partly numerical study on traveling wave solutions where the linear coefficient $g$ is a periodic arrangement of delta potentials. Another rigorous existence results for breathers on finite but large time scales can be found in \cite{DohnalSchnaubeltTietz} for a set-up of Kerr-nonlinear dielectrics occupying two different halfspaces. In our recent paper \cite{Ohrem_Reichel} we proved the first (to the best of our knowledge) existence result for polychromatic breathers in the context of nonlinear Maxwell's equations without presence of any delta-potentials. The $\chi^{(1)}$-part of the polarization was instantaneous and the $\chi^{(3)}$-part was compactly supported in space and either instantaneous or retarded. Due to the compact support in space both variants of the nonlinearity could be treated with the same variational method. Beyond this result we are not aware of any rigorous treatment of polychromatic breathers in the context of nonlinear Maxwell's equations with time retarded material laws.

\subsection{Examples} \label{subsec:examples}
In two theorems below, we give examples of susceptibility tensors $\chi^{(1)}, \chi^{(3)}$ for which breather solutions of Maxwell's equations \eqref{eq:maxwell},\eqref{eq:material},\eqref{eq:polarization} exist. 
These examples are special cases of a general existence result given later in this chapter, cf. \cref{thm:main}.
Let us note that in contrast to some of the previously mentioned results, our breather solutions are generally polychromatic in nature and the potentials considered are bounded functions. Since our breathers lie in suitable Sobolev spaces they are sufficiently differentiable to solve Maxwell's equations pointwise, and they decay at infinity in an $L^p$-sense. 
They may have higher-order space-derivatives depending on smoothness of the material coefficients in space. They are also infinitely differentiable in time because the material properties do not change over time.  

We begin with an exemplary result for the slab geometry \eqref{eq:material:slab}. 

\begin{theorem}\label{thm:ex:slab}
	Let $T > 0$ denote the temporal period, $\omega \coloneqq \frac{2 \pi}{T}$ the associated frequency, and $c \in (0, 1)$ the speed of travel of the breather solution. Assume that the linear susceptibility tensor is given by $\chi^{(1)}(\bfx, \tau) = g(x) \delta(\tau)I$, and the nonlinear susceptibility tensor $\chi^{(3)}$ is given by \eqref{eq:material:noni} or \eqref{eq:material:nonii} with 
	\begin{align*}
		h(\bfx) = h(x),
		\qquad
		\nu(\tau) = (2 - \abs{\sin(\omega t)}) \bbone_{[0, T]}(\tau).
	\end{align*}
	Moreover, assume that the potentials $g, h \in L^\infty(\R)$ have $X$-periodic backgrounds $g^\per, h^\per \in L^\infty(\R)$ such that
	\begin{align*}
		g(x) - g^\per(x) \to 0,
		\qquad 
		h(x) - h^\per(x) \to 0 \quad \text{ as } x\to \pm \infty
	\end{align*}
	and the inequalities
	\begin{align*}
		g^\per \leq g,
		\qquad
		\esssup_\R g < \tfrac{1}{c^2} - 1, 
		\qquad
		h^\per \leq h,
		\qquad
		h^\per \not \leq 0
	\end{align*}
	are satisfied. Then there exist nonzero time-periodic solutions $\bfD, \bfE, \bfB, \bfH$ of Maxwell's equations \eqref{eq:maxwell},\eqref{eq:material},\eqref{eq:polarization} where $\bfE$ is of the form \eqref{eq:ansatz:slab}.
	They satisfy 
	\begin{align*}
		\partial_t^n \bfE \in W^{2,p}(\Omega; \R^3), 
		\qquad
		\partial_t^n \bfB, \partial_t^n \bfH \in W^{1,p}(\Omega; \R^3),
		\qquad
		\partial_t^n \bfD \in L^{p}(\Omega; \R^3)
	\end{align*}
	for all $n \in \N_0$, $p \in [2, \infty]$ and all domains $\Omega = \R \times [y, y+1] \times [z, z+1] \times [t, t+1]$, with norm bounds independent of $y, z, t$.
\end{theorem}

The potentials $g, h$ describe the spatial dependency of the polarization field.
In the above theorem we have required them to be asymptotically periodic at $\pm \infty$. 
This periodic structure helps us to overcome noncompactness of embeddings on $\R$.
The assumption on the ordering $g \geq g^\per, h \geq h^\per$ is a standard tool to resolve noncompactness issues also for the nonperiodic problem.
The upper bound $\tfrac{1}{c^2} - 1$ on $g$ and the choice of $\nu$  ensure that \eqref{eq:2nd_order_maxwell} is elliptic. One aspect of the choice of $\nu$ is that its Fourier coefficients are positive. This aspect will become very important in the general result of Theorem~\ref{thm:main}. 
Ellipticity will ensure that the associated energy has a mountain-pass geometry, and a mountain-pass method will be used to construct breather solutions.
Note also that breathers are localized in the $x$-direction (in the $L^p$-sense stated above), but not in $y, z,$ or $t$, which is due to the ansatz \eqref{eq:ansatz:slab}, since all solutions satisfying this ansatz necessarily are independent of $y$ and periodic in both $z$ and $t$.

Similar to \cref{thm:ex:slab} for the slab geometry, below we give an exemplary result with cylindrical geometry \eqref{eq:material:cyl}.

\begin{theorem}\label{thm:ex:cyl}
	Let $T > 0$ be the period of the breather, $c \in (0, 1)$ be its speed, and $g, h \in L^\infty([0, \infty))$ be material coefficients. Define the linear susceptibility by $\chi^{(1)}(\bfx, \tau) \coloneqq g(r)\delta(\tau) I$ and let the nonlinear susceptibility $\chi^{(3)}$ be given by \eqref{eq:material:noni} or \eqref{eq:material:nonii} with $h(\bfx) = h(r), \nu(\tau) = (2 - \abs{\sin(\omega t)})\bbone_{[0, T]}(\tau)$ where $r \coloneqq \sqrt{x^2 + y^2}$,  $\omega \coloneqq \frac{2 \pi}{T}$. Further let 
	\begin{align*}
		\esssup_\R g < \tfrac{1}{c^2} - 1, 
		\qquad
		h\not\leq0.
	\end{align*}
	Then there exist nonzero time-periodic solutions $\bfD, \bfE, \bfB, \bfH$ of Maxwell's equations \eqref{eq:maxwell},\eqref{eq:material},\eqref{eq:polarization} where $\bfE$ is of the form \eqref{eq:ansatz:cyl}.
	They satisfy 
	\begin{align*}
		\partial_t^n \bfE \in W^{2,p}(\Omega; \R^3), 
		\qquad
		\partial_t^n \bfB, \partial_t^n \bfH \in W^{1,p}(\Omega; \R^3),
		\qquad
		\partial_t^n \bfD \in L^{p}(\Omega; \R^3)
	\end{align*}
	for all $n \in \N_0$, $p \in [2, \infty]$ and all domains $\Omega = \R^2 \times [z, z+1] \times [t, t+1]$, with norm bounds independent of $z, t$.
\end{theorem}

In contrast to \cref{thm:ex:slab}, in \cref{thm:ex:cyl} we do not need any asymptotics for the potentials $g, h$. This is because the cylindrical setting itself comes with compactness, as we discuss in \cref{sec:cylinder_modifications}. To illustrate this, recall that the Sobolev embedding $H^1(\R^2) \embeds L^p(\R^2)$ for $p \in [2, \infty)$ becomes compact when restricted to radially symmetric functions. Lastly, the ansatz \eqref{eq:ansatz:cyl} is periodic in both $z$ and $t$, so breather solutions in the cylindrical setting decay in the $x$ and $y$ directions orthogonal to the direction of propagation.

\subsection{Main theorem}
Before stating the main theorem, we fix some notation.

\subsubsection{Measures on torus and real line, periodic reduction of a measure}
Since breathers are time-periodic, the natural time domain is the torus $\T \coloneqq \R/_{T \Z}$ with period $T$, and we denote the canonical projection by $\projT \colon \R \to \T$. With $\calM(\T)$, $\calM(\R)$ we denote the set of all $\R$-valued measures on the Borel $\sigma$-algebra of $\T$ and $\R$, respectively. 
The push-forward map $\projT^*: \calM(\R)\to \calM(\T)$ is defined as follows: for $\lambda\in \calM(\R)$ we set $\projT^*(\lambda)\in \calM(\T)$ by $\projT^*(\lambda)(E) = \lambda(\projT^{-1}(E))$ for any Borel subset $E\subseteq \T$. The new measure $\projT^*(\lambda)\in\calM(\T)$ is called the periodic reduction of $\lambda$. In this way, the torus is equipped with the measure $\Der t = \frac{1}{T}\projT^*(\bbone_{[0,T]}\der \tau)$, where $\der \tau$ denotes the Lebesgue measure on $\R$.

\subsubsection{Instantaneous vs. retarded $\chi^{(1)}$-contribution} While the nonlinear susceptibility tensor $\chi^{(3)}$ necessarily represents a retarded material response, cf. \eqref{eq:material:noni} or \eqref{eq:material:nonii}, the $\chi^{(1)}$-contribution to the material response may be instantaneous or retarded. The first case is given by $\chi^{(1)}(\bfx,\tau)=g(x)\delta(\tau)I$ or $\chi^{(1)}(\bfx,\tau)=g(r)\delta(\tau)I$ from \cref{subsec:examples}. The second case may be written in the form $\chi^{(1)}(\bfx,\tau)\der \tau = \Der G(\bfx)(\tau) I$ where for fixed $\bfx\in\R^3$ we have that $G(\bfx)\in \calM(\R)$ is an $\R$-valued Borel measure. Mathematically, the second case comprises the first and hence in the following an instantaneous $\chi^{(1)}$-contribution is subsumed in the retarded case.  

\subsubsection{Fourier transform}
Let us fix a convention for Fourier series and Fourier transform. For a time-periodic function $v \colon \T \to \C$ we define its Fourier coefficients by $\hat v_k = \F_k[v] = \int_\T v \overline{e_k} \der t$ with $e_k(t) \coloneqq \ee^{\ii k \omega t}$, $\omega \coloneqq \frac{2 \pi}{T}$. Thus the inverse is $v(t) = \F^{-1}_t[\hat v_k] = \sum_{k \in \Z} \hat v_k e_k(t)$. For a function $v$ depending on space and $T$-periodically on time, $\hat v$ will always denote the (discrete) Fourier transform in time. In the same way we define the discrete Fourier transform in time $\hat\lambda$ of a measure $\lambda\in \calM(\T)$. Finally, a function $v:\T\to\R$ or a measure $\lambda\in\calM(\T)$ is called  positive definite if the sequence $\hat v=(\hat v_k)_{k\in\Z}$ or $\hat\lambda=(\hat \lambda_k)_{k\in\Z}$, respectively, consists of nonnegative entries.

\medskip

Similarly we fix the notion of the spatial (continuous) Fourier transform of a space-dependent function $v \colon \R^d \to \C$, writing $\F_\xi[v] = \int_{\R^d} v(x) \overline{\ee^{\ii x \cdot \xi}} \derf[d] x$ with inverse $\F^{-1}_x[v] = \int_{\R^d} v(\xi) \ee^{\ii x \cdot \xi} \derf[d] \xi$. The spatial (continuous) Fourier transform of a function depending on both space and time is defined analogously, and we omit indices of $\F, \F^{-1}$ when they are clear from the context.

\subsubsection{Cylindrical and slab geometry} We say that a map $A:\R^3\to Y$ possesses cylindrical symmetry if $A(\bfx)=A(\bftx)$ for all $\bfx=(x,y,z), \bftx=(\tilde x,\tilde y,\tilde z)\in\R^3$ with $x^2+y^2={\tilde x}^2+{\tilde y}^2$. In this case we write $A(\bfx)=A(r)$ with $r=\sqrt{x^2+y^2}$. Likewise we say that a map $A:\R^3\to Y$ possesses slab symmetry if $A(\bfx)=A(\bftx)$ for all $\bfx=(x,y,z), \bftx=(x,\tilde y,\tilde z)\in\R^3$ and write $A(\bfx) = A(x)$ in this case.

\begin{theorem}\label{thm:main}	
	Let $T > 0$ denote the temporal period, $\omega \coloneqq \frac{2 \pi}{T}$ the associated frequency, and $c \in (0, 1)$ the speed of travel of the breather solution. We make the following assumptions:
	\begin{enumerate}
        \item[(A1)]\label{ass:first} The linear susceptibility tensor $\chi^{(1)}$ is given by $\chi^{(1)}(\bfx, \tau)\der\tau = \der G(\bfx)(\tau) I$ where $G\colon \R^3 \to \calM(\R)$ is measurable. The nonlinear susceptibility tensor  $\chi^{(3)}$ is given by \eqref{eq:material:noni} or \eqref{eq:material:nonii} where $h\in L^\infty(\R^3)$ and $\nu\in \calM(\R)$. 

        \item[(A2)] $G$ and $h$ both have either cylindrical or slab geometry.
        
        \item[(A3)] $\sup_{\bfx \in \R^3} \norm{G(\bfx)}_{\calM(\R)} < \infty$ and $h \not \leq 0$.
        
		\item[(A4)] \label{ass:upper:bound}
		The periodic reduction $\calG(\bfx)$ of $G(\bfx)$ is even in time for all $\bfx\in \R^3$ and satisfies $\sup_{\substack{\bfx \in \R^3, k \in \Z}} \F_k[\calG(\bfx)] < \tfrac{1}{c^2} - 1$.
		
		\item[(A5)]\label{ass:nonlin:decay} The periodic reduction $\calN$ of $\nu$ is even in time, $\not =0$, and $\abs{k}^{-\beta} \lesssim \F_k[\calN] \lesssim \abs{k}^{-\alpha}$ for all $k \in \Z\setminus\{0\}$ with $\F_k[\calN] \not =0$ and some $\beta \geq \alpha > \alpha^\star$ where $\alpha^\star = 1$ in the slab geometry and $\alpha^\star = \tfrac{3}{2}$ in the cylindrical geometry. 

		\item[(A6)]\label{ass:nonlin}\label{ass:last} In case of the slab geometry, one of the following holds in addition:
		\begin{enumerate}
			\item[(A6a)]\label{ass:nonlin:loc} $h(x) \to 0$ as $x \to \pm\infty$,
			\item[(A6b)]\label{ass:nonlin:per+loc}
			$\calG(x) = \calG^\per(x) + \calG^\loc(x)$ and $h(x) = h^\per(x) + h^\loc(x)$ where $\calG^\per(x), h(x)$ are periodic with common period, and we have $\norm{\calG^\loc(x)}_{\calM(\T)} \to 0$ and $h^\loc(x) \to 0$ as $x \to \pm \infty$. Moreover, $\calG^\loc(x)$ is positive definite for all $x \in \R$ and $h^\loc \geq 0, h^\per \not\leq 0$ hold.
		\end{enumerate}
	\end{enumerate}
	Under these assumptions, there exists a nontrivial breather solution $\bfD, \bfE, \bfB, \bfH$ of Maxwell's equations \eqref{eq:maxwell},\eqref{eq:material}. It satisfies 
	\begin{align*}
	\partial_t^n \bfE \in W^{2,p}(\Omega; \R^3), 
	\qquad
	\partial_t^n \bfB, \partial_t^n \bfH \in W^{1,p}(\Omega; \R^3),
	\qquad
	\partial_t^n \bfD \in L^{p}(\Omega; \R^3)
	\end{align*}
	for all $n \in \N_0$, $p \in [2, \infty]$ and all domains $\Omega$ that are of the form $\Omega = \R \times [y, y+1] \times [z, z+1] \times [t, t+1]$ in the slab case and $\Omega = \R^2 \times [z, z+1] \times [t, t+1]$ in the cylindrical case, with norm bounds independent of $y, z, t$.
\end{theorem}

\begin{remark} Our assumptions \ref{ass:first}--\ref{ass:nonlin:decay} on the structure of the linear and nonlinear retardation kernels can be seen as a systematic attempt to find out what can be done in a variational setting. The main assumptions are expressed via the Fourier coefficients of $\calG(\bfx)$ and $\calN$. The fact that both $\calG(\bfx)$ and $\calN$ have real Fourier coefficients stems from their evenness in time which could be understood physically as a balance of loss and gain within one period of time. In assumption \ref{ass:upper:bound} the upper bound on $\F_k[\calG(\bfx)]$ can be achieved by taking a sufficiently small propagation speed $c$. It is exactly this assumption which makes the linear operator in \eqref{eq:example} elliptic, and combined with positive definiteness of $\calN$ it makes the quadratic part in \eqref{eq:example_functional} positive definite.
\end{remark}

\begin{remark} \label{rem:infinitely} 
	If in the setting of \cref{thm:main} the set $\regularK \coloneqq \set{k \in \Z\setminus\set{0} \colon \F_k[\calN] \neq 0}$ is infinite then we moreover have existence of infinitely many nontrivial breathers with the stated properties. 
\end{remark}

\begin{remark}
    The sign assumptions on $\calG^\loc, h^\loc$ in \ref{ass:nonlin:per+loc} of \cref{thm:main} yield a strict relation between the mountain-pass energy level of the problem compared to the energy of the ``periodic'' problem, (i.e. with $\calG, h$ replaced by $\calG^\per, h^\per$), see \cref{lem:compare:energy} for a precise formulation. This energy inequality gives us some compactness which is crucial for showing existence of breathers. The examples \cref{thm:ex:slab,thm:ex:cyl} satisfy \ref{ass:first}--\ref{ass:last}. For \ref{ass:nonlin:decay} this is true because
	\begin{align*}
		\F_k[\calN] = \F_k[2 - \abs{\sin(\omega t)}] = \begin{cases}
			2 - \frac{2}{\pi}, & k = 0,
			\\ 0, & k \text{ odd}, 
			\\ \frac{2}{\pi (k^2 - 1)}, & k \neq 0 \text{ even}.
		\end{cases}
	\end{align*}
\end{remark}

Breather solutions are more regular when the material coefficients $\calG, h$ have higher regularity. For $\Omega \subseteq \R^4$ we denote by $C_b^j(\Omega; \R^3)$ the space of $j$-times differentiable functions mapping into $\R^3$ with bounded derivatives, and abbreviate $\tilde C^j_b(\Omega; \R^3) \coloneqq W^{j, 2}(\Omega; \R^3) \cap C^j_b(\Omega; \R^3)$.

\begin{corollary}	
	If in the context of Theorem~\ref{thm:main} we additionally have
	\begin{enumerate}
		\item[(R)]\label{ass:regularity} 
		$g \in C_b^l(\R^3; \calM(\R)), h \in C_b^l(\R^3)$ for some $l \in \N_0$,
	\end{enumerate}
	then the regularity improves to 
	\begin{align*}
		\partial_t^n \bfE \in \tilde C_b^{2+l}(\Omega; \R^3),
		\quad  
		\partial_t^n \bfB, \partial_t^n \bfH \in \tilde C_b^{1+l}(\Omega; \R^3),
		\quad
		\partial_t^n \bfD \in \tilde C_b^{l}(\Omega; \R^3)
	\end{align*}
	with norm bounds independent of $y, z, t$. 
\end{corollary}

\subsection{Outline of paper}

We begin by investigating the slab geometry \eqref{eq:material:slab}. In \cref{sec:variational_problem} we convert Maxwell's equations into the Euler-Lagrange equation of a suitable Lagrangian functional, and show that this functional has mountain-pass geometry. 
Using the mountain-pass theorem, in \cref{sec:ground_state_existance} we show that the Euler-Lagrange equation admits a ground state solution. 
The convergence of Palais-Smale sequences approaching the ground state level is unclear in general because the spatial domain is unbounded, and thus our arguments depend on the particular form of the potentials in \ref{ass:nonlin}. For \ref{ass:nonlin:loc}, the nonlinearity is compact which makes this the easiest case. For \ref{ass:nonlin:per+loc} we first rely on translation arguments in space for the purely periodic case. Then we use comparison arguments for the perturbed periodic case. After having shown existence and multiplicity of breathers, we investigate their regularity in \cref{sec:regularity}. Finally, \cref{sec:cylinder_modifications} details the arguments for the cylindrical geometry \eqref{eq:material:cyl} and highlights the differences to the slab geometry.


\section{variational problem}
\label{sec:variational_problem}

From now on, we always assume that the assumptions of \cref{thm:main} are satisfied. We transform \eqref{eq:profile_problem:slab} into a problem for a surrogate variable $u$, which we then treat using the mountain pass method. 
We only consider the slab problem \eqref{eq:profile_problem:slab} as the cylindrical problem \eqref{eq:profile_problem:cyl} can be treated similarly. 
In \cref{sec:cylinder_modifications} we discuss the differences between the slab and the cylindrical problem, and how to treat the latter. 

Using the periodic reduction $\calG(x), \calN$ of $g(x),  \nu$ we can rewrite the scalar polarization \eqref{eq:scalar_polarization:noni} as
\begin{align*}
	P(w)(x,t)
	&= \epsilon_0 \int_0^\infty w(x, t - \tau) \der g(x)(\tau)
	+ \epsilon_0 h(x) \int_0^\infty w(x, t - \tau)^3 \der \nu(\tau)
	\\ &= \epsilon_0 \int_\T w(x, t - \tau) \der \calG(x)(\tau)
	+ \epsilon_0 h(x) \int_\T w(x, t - \tau)^3 \der \calN(\tau)
\intertext{since $w$ is $T$-periodic in $t$. We abbreviate this by writing}
	P(w) 
	&= \epsilon_0 (\calG \ast w) + \epsilon_0 h (\calN \ast w^3)
\end{align*}
where $\ast$ denotes convolution of a measure with a function on $\T$. Similarly, the polarization \eqref{eq:scalar_polarization:nonii} can be written in the form 
\begin{align*}
	P(w) 
	&= \epsilon_0 (\calG \ast w) + \epsilon_0 h (\calN \ast w)^3,
\end{align*}
Next we define the projection $\projR$ onto the set $\regularK\coloneqq \set{k \in \Z\setminus\set{0} \colon \F_k[\calN] \neq 0}$ of ``regular'' frequency indices by
\begin{align*}
	\projR[v] = \F^{-1}\bigl[\bbone_{k \in \regularK} \F_k[v]\bigr],
\end{align*} 
as well as the projection onto the ``singular'' frequency indices $\singularK \coloneqq \Z \setminus \regularK$ by $\projS[v] \coloneqq \F^{-1}\bigl[\bbone_{k \in \singularK} \F_k[v]\bigr] = (I - \projR)[v]$. We apply both to \eqref{eq:profile_problem:slab} for time-periodic $w$ to obtain the two problems
\begin{align*}
	&\left( - \partial_x^2 + \partial_t^2 (1 - \tfrac{1}{c^2} + \calG\ast) \right)\projR[w]
	+ h \partial_t^2 \projR[N(w)] = 0,
	\\ &\left( - \partial_x^2 + \partial_t^2 (1 - \tfrac{1}{c^2} + \calG\ast) \right)\projS[w]
	+ h \partial_t^2 \projS[N(w)] = 0,
\end{align*}
where the cubic nonlinearity $N(w)$ is given by
\begin{align*}
	N(w) = \calN \ast w^3
	\qquad\text{or}\qquad
	N(w) = (\calN \ast w)^3
\end{align*}
corresponding to \eqref{eq:material:noni} and \eqref{eq:material:nonii}, respectively.

Let us first consider the nonlinearity $N(w) = \calN \ast w^3$. Using $\projS (\calN\ast) = 0$ and assuming that the linear operator $\left( - \partial_x^2 + \partial_t^2 (1 - \tfrac{1}{c^2} + \calG\ast) \right)$ is injective, we can further simplify this to 
\begin{align*}
	\left( - \partial_x^2 + \partial_t^2 (1 - \tfrac{1}{c^2} + \calG\ast) \right) w
	+ h \partial_t^2 (\calN \ast w^3) = 0, 
	\qquad 
	\projS[w] = 0.
\end{align*}
Observe that the convolution operator $\calN\ast$ is formally invertible on $\ker \projS$ since $\F_k[\calN] \neq 0$ for $k \in \regularK$. Therefore we may rephrase this problem as
\begin{align}\label{eq:effective_problem}
	(- \partial_t^2 \calN\ast)^{-1}\left( - \partial_x^2 + \partial_t^2 (1 - \tfrac{1}{c^2} + \calG\ast) \right) u
	- h \projR[u^3] = 0, 
	\qquad 
	\projS[u] = 0
\end{align}
with $u \coloneqq w$.

For the second nonlinearity $N(w) =  (\calN \ast w)^3$ we set $u \coloneqq \calN \ast w$ and therefore get
\begin{align}\label{eq:w_reconstruction:ii}
	\begin{split}
		&(-\partial_t^2 \calN\ast)^{-1} \left( - \partial_x^2 + \partial_t^2 (1 - \tfrac{1}{c^2} + \calG\ast) \right)u
		- h \projR[u^3] = 0,
		\qquad \projS[u] = 0,
		\\ &\left( - \partial_x^2 + \partial_t^2 (1 - \tfrac{1}{c^2} + \calG\ast) \right)\projS[w]
		+ h \partial_t^2 \projS[u^3] = 0.	
	\end{split}
\end{align}
Note that the first of the two equations above is \eqref{eq:effective_problem}. Hence, also for the second nonlinearity, it is sufficient to solve \eqref{eq:effective_problem} for $u$ and then use the second equation to determine the missing values of $\calF_k[w]$ for $k\in\singularK$.

We first focus our attention on investigating the ``effective problem'' \eqref{eq:effective_problem}, which using
\begin{align*}
	V(x) \coloneqq \tfrac{1}{c^2} - 1 - \calG\ast
\end{align*}
we can write as
\begin{align}\label{eq:effective_problem:small}
	(- \partial_t^2 \calN\ast)^{-1}\left( - \partial_x^2 - V(x)\partial_t^2 \right) u - h \projR[u^3] = 0,
	\qquad \projS[u] = 0.
\end{align}
Since $\calG, \calN$ are even in time, the differential operator above is symmetric, and therefore solutions of \eqref{eq:effective_problem:small} formally are critical points of the functional
\begin{align*}
	\fc(u) = \int_{\R \times \T} \left(\tfrac12 u \cdot (- \partial_t^2 \calN\ast)^{-1}\left( - \partial_x^2 - V(x)\partial_t^2 \right) u - \tfrac14 h(x) u^4 \right) \der (x, t),
	\qquad \projS[u] = 0.
\end{align*} 

Next we properly define the domain $\fcd$ of the functional $\fc$ sketched above, and we investigate embeddings $\fcd \embeds L^p$.

\begin{definition}\label{def:slab:functional_domain}
	We define the space
	\begin{align*}
		\fcd \coloneqq \set{u \in L^2(\R \times \T) \colon \hat u_k \equiv 0 \text{ for } k \in \Z \setminus \regularK, \norm{u}_\fcd^2 \coloneqq \iip{u}{u}_{\fcd} < \infty}
	\end{align*}
	where
	\begin{align*}
		\iip{u}{v}_\fcd = \sum_{k \in \regularK} \frac{1}{\omega^2 k^2 \F_k[\calN]} 
		\int_\R \left(\hat u_k' \overline{\hat v_k'} + \omega^2 k^2 \hat u_k \overline{\hat v_k} \right) \der x
	\end{align*}
	Note that $V_k(x) \coloneqq \tfrac{1}{c^2} - 1 - \F_k[\calG(x)]$ is bounded and strictly positive by assumption, so that
	\begin{align*}
		\ip{u}{v}_\fcd \coloneqq \sum_{k \in \regularK} \frac{1}{\omega^2 k^2 \F_k[\calN]} \int_\R \bigl(\hat u_k' \overline{\hat v_k'} + \omega^2 k^2 V_k(x) \hat u_k \overline{\hat v_k}\bigr) \der x
	\end{align*}
	defines an equivalent inner product on $\fcd$.
\end{definition}

Note that $\fcd$ is a Hilbert space since $\fcd \embeds L^2$. Assumption~\ref{ass:nonlin:decay} on the decay of the Fourier coefficients of $\calN$ ensures that $\fcd$ continuously embeds into $L^p(\R \times \T)$ for all $p \in [2, p^\star)$ where $p^\star > 4$, as we show below in \cref{lem:Lp_embedding}. This allows us to write 
	\begin{align*}
		\fc(u) = \tfrac12 \ip{u}{u}_\fcd - \tfrac14 \int_{\R \times \T} h(x) u^4 \der (x,t)
		\quad\text{for}\quad
		u \in \fcd,
	\end{align*}
and to define solutions $u \in \fcd$ of \eqref{eq:effective_problem} in the following way.

\begin{definition}[weak solution]\label{def:slab:weak_solution}
	A function $u \colon \R \times \T \to \R$ is called a \emph{weak solution} to \eqref{eq:effective_problem} if $u \in \fcd$ and 
	\begin{align*}
		\ip{u}{v}_\fcd - \int_{\R \times \T} h(x) u^3 v \der (x,t) = 0
	\end{align*}
	for all $v \in \fcd$. This is equivalent to $\fc'(u) = 0$.
\end{definition}

It is standard to verify the validity of the following density result for $\fcd$, which will prove very useful for some approximation arguments.
\begin{lemma}\label{lem:dense_subset}
	The set $\set{u \in C_c^\infty(\R \times \T) \colon \hat u_k \equiv 0 \text{ for almost all } k \in \Z} \cap \fcd$ is dense in $\fcd$.
\end{lemma}

\begin{lemma}\label{lem:Lp_embedding}
	For any $p \in [2, p^\star)$ with $p^\star \coloneqq \frac{4}{2 - \alpha}$ ($p^\star = \infty$ if $\alpha \geq 2$), the embedding $\fcd \embeds L^p(\R \times \T)$ is continuous and the embedding $\fcd \embeds L^p_\loc(\R \times \T)$ is compact.
\end{lemma}
\begin{proof}
	Let us first show continuity. For this, we calculate
	\begin{align*}
		\norm{u}_p
		\lesssim \norm{\F_{\xi, k}[u]}_{p'}
		\lesssim \norm{\sqrt{\frac{\omega^2 k^2 \F_k[\calN]}{\xi^2 + \omega^2 k^2}}}_r 
		\cdot \norm{\sqrt{\frac{\xi^2 + \omega^2 k^2}{\omega^2 k^2 \F_k[\calN]}} \F_{\xi, k}[u]}_2
		\lesssim \norm{u}_\fcd
	\end{align*}
	where $\frac{1}{r} = \frac{1}{2} - \frac{1}{p} < \frac{\alpha}{4}$ and where by \ref{ass:nonlin:decay} the first term has been estimated as follows 
	\begin{align*}
		\MoveEqLeft\norm{\sqrt{\frac{\omega^2 k^2 \F_k[\calN]}{\xi^2 + \omega^2 k^2}}}_r^r
		= \sum_{k \in \regularK} \int_\R \left( \frac{\omega^2 k^2 \F_k[\calN]}{\xi^2 + \omega^2 k^2} \right)^{\nicefrac r2} \der \xi
		\\ &= \int_\R \left( \frac{1}{\xi^2 + 1} \right)^{\nicefrac r2} \der \xi \cdot \sum_{k \in \regularK} \abs{\omega k} \F_k[\calN]^{\nicefrac{r}{2}}
		\lesssim \sum_{k \in \regularK} \abs{k}^{1 - \frac{\alpha r}{2}} < \infty.
	\end{align*}
    In order to show compactness, define for $K \in \Nodd$ the projection onto ``small'' frequencies $P_K \colon \fcd \to \fcd$ by $P_K[u] = \F_t^{-1}\left[\bbone_{\abs{k} \leq K} \F_k[u]\right]$. Then on $P_K \fcd$ the norm $\norm{u}_\fcd$ is equivalent to
	\begin{align*}
		\nnorm{u} = \sum_{\substack{k \in \regularK \\ \abs{k} \leq K}} \norm{\hat u_k}_{\fcd^1(\R)}.
	\end{align*}
	Since the embedding $\fcd^1(\R) \to L^p_\loc(\R)$ is compact for any $p \in [2, \infty]$ and the sum above is finite, it follows that $P_K \colon \fcd \to L^p_\loc(\R \times \T)$ is compact. Next, the calculations above show for $u \in \fcd$ that 
	\begin{align*}
		\norm{u - P_K[u]}_p \leq C \sum_{\substack{k \in \regularK \\ \abs{k} > K}} \abs{k}^{1 - \frac{\alpha r}{2}} \norm{u}_\fcd
	\end{align*}
	for some $C > 0$ independent of $K$, so that $P_K \to I$ in $\calB(\fcd; L^p(\R \times \T))$ as $K\to\infty$. Thus $\fcd$ embeds compactly into $L^p_\loc(\R \times \T)$.
\end{proof}

\begin{definition}
	For $s \in \R$ we define the fractional time-derivative $\fracDT{s}$ as the Fourier multiplier $\fracDT{s} v(t) = \F_t^{-1}[\abs{\omega k}^s \hat v_k]$.
\end{definition}

\begin{corollary}\label{cor:Lp_embedding:derivatives}
	As in the proof of \cref{lem:Lp_embedding} we see that for $p \in [2, p^\star)$ and $\eps > 0$ sufficiently small (depending on $\alpha, p$), the map $\fracDT{\eps} \colon \fcd \to L^p(\R \times \T)$ is bounded and $\fracDT{\eps} \colon \fcd \to L^p_\loc(\R \times \T)$ is compact.
\end{corollary}

Let us recall the notion of a Palais-Smale sequence.
\begin{definition}
	A sequence $(u_n)$ in $\fcd$ is called a Palais-Smale sequence for $\fc$ if $\fc'(u_n) \to 0$ in $\fcd'$ and $\fc(u_n)$ converges in $\R$ as $n\to\infty$.  If $\lim_{n\to\infty} \fc(u_n) = c$, we call $(u_n)$ a Palais-Smale sequence at level $c$.
\end{definition}

In the following lemma we show a variant of the concentration-compactness principle that will be a useful tool for extracting a nonzero limit from Palais-Smale sequences. 

\begin{lemma}\label{lem:concentration-compactness}
	Let $(u_n)$ be a bounded sequence in $\fcd$, $r > 0$ and $\tilde p \in [2, p^\star)$ such that
	\begin{align*}
		\sup_{x \in \R} \norm{u_n}_{L^{\tilde p}([x - r, x + r] \times \T)} \to 0
	\end{align*}
	as $n \to \infty$. Then $u_n \to 0$ in $L^p(\R \times \T)$ for all $p \in (2, p^\star)$.
\end{lemma}

\begin{proof}
	By Hölder's inequality and \cref{lem:Lp_embedding} it suffices to show the result for $\tilde p = 2$ and one $p \in (2, p^\star)$, which we shall choose so close to $2$ that $2 < q \coloneqq \frac{4}{4 - p} < p^\star$. Let $\phi_m \colon \R \to [0,1]$ be a smooth partition of unity with $\supp \phi_m \subseteq [(m-1)r, (m+1)r]$, $\norm{\phi_m'}_\infty \leq \frac{2}{r}$. Using that at any point of $\R$ at most $2$ of the $\phi_m$ are nonzero, we calculate
	\begin{align*}
		\norm{u_n}_p^p 
		&= \int_{\R \times \T} \abs{\sum_{m \in \Z} \phi_m u_n}^p \der (x, t)
		\\ &\leq 2^{p-1} \int_{\R \times \T} \sum_{m \in \Z} \abs{\phi_m u_n}^p \der (x, t)
		\\ &= 2^{p-1} \sum_{m \in \Z} \norm{\phi_m u_n}_p^p
		\\ &\leq 2^{p-1} \sum_{m \in \Z} \norm{\phi_m u_n}_{q}^2 \norm{\phi_m u_n}_{2}^{p-2}
		\\ &\lesssim \sup_{x \in \R} \norm{u_n}_{L^2([x-r, x+r] \times \T)}^{p-2} \sum_{m \in \Z} \norm{\phi_m u_n}_\fcd^2.
	\end{align*}
	Moreover, since
	\begin{align*}
		\norm{\phi_m u_n}_\fcd^2
		&= \sum_{k \in \regularK} \frac{1}{\omega^2 k^2 \F_k[\calN]} 
		\int_\R \left(\abs{\phi_m' \hat u_k + \phi_m \hat u_k'}^2 + \omega^2 k^2 \abs{\phi_m \hat u_k}^2 \right) \der x
		\\ &\leq C
		\sum_{k \in \regularK} \frac{1}{\omega^2 k^2 \F_k[\calN]} 
		\int_{(m-1)r}^{(m+1)r} \left(\abs{\hat u_k'}^2 + \omega^2 k^2 \abs{\hat u_k}^2 \right) \der x,
	\end{align*}
	it follows that $\sum_{m \in \Z} \norm{\phi_m u_n}_\fcd^2 \leq 2 C \norm{u}_\fcd^2$. Thus, from the assumptions we obtain $\norm{u_n}_p \to 0$ as $n \to \infty$.
\end{proof}


\section{Existence of ground states}
\label{sec:ground_state_existance}

In the following, let $\fc$ be given by \cref{def:slab:weak_solution}. We call the energy level 
\begin{align*}
	c_\gs \coloneqq \inf_{\substack{u \in \fcd \setminus \set0 \\ \fc'(u) = 0}} \fc(u)
\end{align*}
the \emph{ground state energy level}, and any solution $u \in \fcd \setminus \set{0}$ of $\fc'(u) = 0$ with $\fc(u) = c_\gs$ a \emph{ground state} of $\fc$. Note that $c_\gs=+\infty$ if there are no nonzero critical points of $\fc$. Next we present the main result of this section. The rest of this section is dedicated to its proof.

\begin{theorem}\label{thm:ground_state}
	There exists a ground state of $\fc$.
\end{theorem}

We first note that the following necessary condition for existence of ground states holds.

\begin{lemma}\label{lem:positive_energy}
	$c_\gs > 0$.
\end{lemma}
\begin{proof}
	By \cref{lem:Lp_embedding} we have $\fc'(u)[u] = \ip{u}{u}_\fcd + \landauO(\norm{u}_\fcd^4)$ as $u \to 0$, so there exists $c > 0$ with $\norm{u}_\fcd \geq c$ for all $u \in \fcd\setminus\set0$ with $\fc'(u) = 0$. The claim follows from this since for every critical point $u$ of $\fc$ we have $\fc(u) = \fc(u) - \tfrac14 \fc'(u)[u] = \tfrac14 \ip{u}{u}_\fcd$.
\end{proof}

We will extract the ground state as a limit of a suitable Palais-Smale sequence. Next we use the mountain-pass theorem to obtain a particular Palais-Smale sequence.

\begin{proposition}\label{prop:mountain_pass}
	There exists $u_0 \in \fcd$ with $\fc(u_0) < 0$. For such $u_0$, the mountain-pass energy level
	\begin{align*}
		c_\mopa \coloneqq 
		\inf_{\substack{\gamma \in C([0;1]; \fcd) \\ \gamma(0) = 0, \gamma(1) = u_0}}
		\sup_{s \in [0, 1]} \fc(\gamma(s))
	\end{align*}
	is positive and there exists a Palais-Smale sequence for $\fc$ at level $c_\mopa$.
\end{proposition}
\begin{proof}
	For the construction of a suitable $u_0$ we choose $\varphi \in C_c^\infty(\R)$ with $\int_\R h \varphi^4 \der x> 0$, which exists since $h \not \leq 0$ and $C_c^\infty(\R)$ is dense in $L^4(\R)$. We then choose $u_0(x,t) = r \Re[\varphi(x) e_{k_0}(t)]$ for some $k_0 \in \regularK$ and $r>0$. This implies that
	\begin{align*}
		\fc(u_0) = \tfrac{1}{2} r^2 \ip{\Re[\varphi(x) e_{k_0}(t)]}{\Re[\varphi(x) e_{k_0}(t)]}_\fcd - \tfrac{3}{32} r^4 \int_\R h \varphi^4 \der x
	\end{align*}
	is negative, provided $r$ has been chosen sufficiently large. By the embedding $\fcd \embeds L^4$ we moreover have
	\begin{align*}
		\fc(u) = \tfrac12 \ip{u}{u}_\fcd - \landauO\left(\norm{u}_\fcd^4\right)
	\end{align*}
	as $u \to 0$. Thus $c_\mopa > 0$ and by the mountain pass theorem, cf. \cite{willem}, there exists a Palais-Smale sequence $(u_n)$ at level $c$. 
\end{proof}

\begin{lemma}\label{lem:ps_bounded}
	Any Palais-Smale sequence for $\fc$ is bounded.
\end{lemma}
\begin{proof}
	Let $(u_n)$ be a Palais-Smale sequence at level $c$. Then 
	\begin{align*}
		\ip{u_n}{u_n}_\fcd = 4 \fc(u_n) - \fc'(u_n)[u_n]
		= 4 c + \landauo(1) + \landauo(\norm{u_n}_\fcd)
	\end{align*}
	as $n \to \infty$, which shows that $(u_n)$ is bounded in $\fcd$.
\end{proof}

Next we show the following result on weakly convergent Palais-Smale sequences in our setting.

\begin{lemma}\label{lem:weakly_convergent_ps}
	Let $(u_n)$ be a Palais-Smale sequence for $\fc$ with $\fc(u_n) \to c$ and $u_n \wto u$ in $\fcd$. Then $u$ is a critical point of $\fc$ and $\fc(u) \leq c$. Moreover, if $u \neq 0$ and $c = c_\gs$ then $u$ is a ground state and $u_n \to u$ in $\fcd$.
\end{lemma}
\begin{proof}
	By \cref{lem:Lp_embedding,lem:ps_bounded} we have $u_n \to u$ in $L^4_\loc$. Thus for compactly supported $v \in \fcd$ it follows that
	\begin{align*}
		\fc'(u_n)[v] = \ip{u_n}{v}_\fcd - \int_{\R \times \T} h(x) u_n^3 v \der (x, t)
		\to \ip{u}{v}_\fcd - \int_{\R \times \T} h(x) u_n^3 v \der (x, t)
		= \fc'(u)[v]
	\end{align*} 
	so that $\fc'(u)[v] = 0$. By a density argument (cf. \cref{lem:dense_subset}) it follows that $u$ is a critical point of $\fc$. Next we calculate
	\begin{align*}
		\fc(u) = \fc(u) - \tfrac14 \fc'(u)[u] = \tfrac14 \ip{u}{u}_\fcd
		\leq \tfrac14 \lim_{n \to \infty} \ip{u_n}{u_n}_\fcd = \lim_{n \to \infty} \fc(u_n) - \tfrac14 \fc'(u_n)[u_n] = c.
	\end{align*}
	If $c = c_\gs$, then we have $\fc(u) \geq c_\gs$ since $u \neq 0$ by assumption, and thus from the above inequality we find $\fc(u) = c_\gs$ and in addition $\ip{u_n}{u_n}_\fcd \to \ip{u}{u}_\fcd$. Combined with $u_n \wto u$ in $\fcd$ this shows $u_n \to u$ in $\fcd$.
\end{proof}

In many situations, e.g. in a translation-invariant setting, there are always Palais-Smale sequences converging weakly to $0$. Therefore the main task in the following will be to find a Palais-Smale sequence with $u_n \wto u \neq 0$. The arguments for this (and the proof of \cref{thm:ground_state}) differ between the types of nonlinearity, and are split into subsections accordingly.

\subsection{Proof of Theorem~\ref{thm:ground_state} for \ref{ass:nonlin:loc} and the purely periodic case of \ref{ass:nonlin:per+loc}}

First we show how to extract a nonzero limit from a given Palais-Smale sequence.

\begin{lemma}\label{lem:ps_convergence_part12}
	Assume \ref{ass:nonlin:loc} or \ref{ass:nonlin:per+loc} with $\calG^\loc, h^\loc\equiv 0$.
	Let $(u_n)$ be a Palais-Smale sequence for $\fc$ at level $c > 0$. Then there exists a critical point $u \in \fcd \setminus\set0$ of $\fc$ with $\fc(u) \leq c$.
\end{lemma}

\begin{proof}
	\emph{Part 1:} We consider \ref{ass:nonlin:loc}. 
	Up to a subsequence we have $u_n \wto u$ in $\fcd$ and $u_n \to u$ in $L^4_\loc$ by \cref{lem:Lp_embedding,lem:ps_bounded}, where \cref{lem:weakly_convergent_ps} guarantees that $u$ is a critical point of $\fc$. Moreover, since $h(x) \to 0$ as $x \to \pm \infty$, we have $h(x) u_n^3 \to h(x) u^3$ in $L^{\nicefrac43}(\R \times \T)$. This implies for $v \in \fcd$ that
	\begin{align*}
		\ip{u_n - u}{v}_\fcd
		= \fc'(u_n)[v] - \fc'(u)[v] + \int_{\R \times \T} h(x) (u^3 - u_n^3) v \der (x, t)
		= \landauo(\norm{v}_\fcd)
	\end{align*}
	as $n \to \infty$. So $u_n \to u$ in $\fcd$, and in particular $\fc(u) = c$ and $u \neq 0$ hold.

	\medskip
	\emph{Part 2:} We now consider \ref{ass:nonlin:per+loc} with $\calG^\loc, h^\loc\equiv 0$.
	Since
	\begin{align*}
		\int_{\R \times \T} h(x) u_n^4
		= 4 \fc(u_n) - 2 \fc'(u_n)[u_n] \to 4 c,
	\end{align*}
	we have $u_n \not \to 0$ in $L^4(\R \times \T)$. Let $X>0$ denote the period of $\calG$ and $h$. By \cref{lem:concentration-compactness} there exist $x_n \in \R$ with
	\begin{align}\label{eq:loc:positive_L4_bound}
		\liminf_{n \to \infty} \norm{u_n}_{L^4([x_n - X, x_n + X] \times \T)} > 0
	\end{align}
	and w.l.o.g. we may assume $x_n \in X \Z$. Let us define a new sequence $\tilde u_n$ by $\tilde u_n(x, t) = u_n(x - x_n, t)$, so that $\fc(\tilde u_n) = \fc(u_n) \to c$ and $\fc'(\tilde u_n) \to 0$. Up to a subsequence we have $\tilde u_n \wto u$ in $\fcd$ where $u \neq 0$ by \eqref{eq:loc:positive_L4_bound}. The claim now follows from \cref{lem:weakly_convergent_ps} applied to $(\tilde u_n)$.
\end{proof}

\begin{proof}[Proof of \cref{thm:ground_state} for \ref{ass:nonlin:loc} and \ref{ass:nonlin:per+loc} with $\calG^\loc, h^\loc\equiv 0$]
	Combining \cref{prop:mountain_pass} and \cref{lem:ps_convergence_part12} we see that there exists a nonzero critical point of $\fc$. Thus $c_\gs < \infty$ and by definition of $c_\gs$ there exists a sequence $(u_n)$ of critical points of $\fc$ with $\fc(u_n) \to c_\gs$. Since $c_\gs > 0$ by \cref{lem:positive_energy}, applying \cref{lem:ps_convergence_part12} to $(u_n)$ we find a ground state of $\fc$.
\end{proof}

\subsection{Proof of Theorem~\ref{thm:ground_state} for \ref{ass:nonlin:per+loc}} 
We call the problem with $\calG, h$ replaced by $\calG^\per, h^\per$ the \emph{periodic} problem and denote it with the superscript ``$\per$''. The previous subsection guarantees the existence of a periodic ground state $u^\per$ of $\fc^\per$.

Note that both $\fc$ and $\fc^\per$ are defined on the same Hilbert space $\fcd$. The assumption \ref{ass:nonlin:per+loc} implies that $\fc \leq \fc^\per$ on $\fcd$ and, assuming $(\calG^\loc, h^\loc) \neq 0$, the inequality is even strict on functions which do not have zero sets of positive measure. For our nonlocal problem \eqref{eq:effective_problem} we do not know whether or not a unique continuation theorem holds, which is why we cannot rule out that a critical point of $\fc$ or $\fc^\per$ could have a zero set of positive measure. Nevertheless, the subsequent arguments work without a unique continuation theorem and are based on the comparison of energy levels between our current and the periodic problem.

\begin{lemma}\label{lem:compare:energy}
	Assume that no ground state of $\fc^\per$ is a critical point of $\fc$.
	Then there exists $u_0 \in \fcd$ with $\fc(u_0) < 0$ such that the mountain-pass energy level
	\begin{align*}
		c_\mopa \coloneqq 
		\inf_{\substack{\gamma \in C([0;1]; \fcd) \\ \gamma(0) = 0, \gamma(1) = u_0}}
		\sup_{s \in [0, 1]} \fc(\gamma(s))
	\end{align*}
	satisfies $0 < c_\mopa < c_\gs^\per$.
\end{lemma}

\begin{proof}
	Let $u^\per$ be a ground state of $\fc^\per$. As $u^\per$ is not a critical point of $\fc$, we have $\calG^\loc \ast u^\per \neq 0$ or $h^\loc (u^\per)^3 \neq 0$. By the assumptions on the signs of $\calG^\loc, h^\loc$ we moreover have
	\begin{align*}
		\ip{u^\per}{u^\per}_\fcd \leq \ip{u^\per}{u^\per}_\fcd^\per
		\quad\text{and}\quad
		\int_{\R\times\T} h(x) (u^\per)^4 \der (x, t) \geq \int_{\R\times\T} h^\per(x) (u^\per)^4 \der (x, t)
	\end{align*}
	where at least one inequality is strict. In particular, $\fc(s u^\per) < \fc^\per(s u^\per)$ holds for $s \neq 0$.
	Now set $u_0 \coloneqq \sqrt{2} u^\per$. Then $\fc(u_0) < \fc^\per(u_0) = 0$ and
	\begin{align*}
		c_\mopa \leq \max_{s \in [0, 1]} \fc(s u_0)
		< \max_{s \in [0, 1]} \fc^\per(s u_0) 
		= \fc^\per(u^\per) = c_\gs^\per.
	\end{align*}
	Positivity of $c_\mopa$ was already shown in \cref{prop:mountain_pass}.
\end{proof}

Similar to \cref{lem:ps_convergence_part12} of the previous subsection, we require a result on convergence of a given Palais-Smale sequence, which we present next.

\begin{lemma}\label{lem:ps_convergence_part3}
	Assume \ref{ass:nonlin:per+loc}.
	Let $u_n$ be a Palais-Smale sequence for $\fc$ at level $c \in (0, c_\gs^\per)$. Then there also exists a critical point $u \in \fcd \setminus\set0$ of $\fc$ with $\fc(u) \leq c$.
\end{lemma}

\begin{proof}
	We denote by $X$ the spatial period of $\calG^\per, h^\per$. As in the proof of \cref{lem:ps_convergence_part12}, Part 2, we have that $u_n \not \to 0$ in $L^4(\R \times \T)$ and that a sequence $x_n \in X \Z$ exists such that 
	\begin{align*}
		\liminf_{n \to \infty} \norm{u_n}_{L^4([x_n - X, x_n + X] \times \T)} > 0.
	\end{align*}
	We claim that $u_n \not \to 0$ in $L^4_\loc$ along any subsequence. 

	Assume for a contradiction that there exists a subsequence of $(u_n)$, which we again denote by $(u_n)$, such that $u_n \to 0$ in $L^4_\loc$. 
	Since $u_n \not\to 0$ in $L^4$, we necessarily have $\abs{x_n} \to \infty$. We define $\tilde u_n$ by $\tilde u_n(x, t) = u_n(x - x_n, t)$. Then up to a subsequence we have $\tilde u_n \wto u$ in $\fcd$ and $\tilde u_n \to u$ in $L^4_\loc$ for some $u \in \fcd \setminus\set0$. For compactly supported $v \in \fcd$ we set $v_n(x, t) = v(x + x_n, t)$ and calculate
	\begin{align*}
		\fc'(u_n)[v_n]
		&= \ip{u_n}{v_n}_\fcd - \int_{\R \times \T} h(x) u_n^3 v_n \der(x, t)
		\\ &= \ip{u_n}{v_n}_\fcd^\per - \int_{\R \times \T} h^\per(x) u_n^3 v_n \der(x, t)
		\\ &\qquad- \sum_{k \in \regularK} \int_\R \frac{\F_k[\calG^\loc(x)]}{\F_k[\calN]} \bigl(\F_k[u_n] \overline{\F_k[v_n]}\bigr) \der x
		- \int_{\R \times \T} h^\loc(x) u_n^3 v_n \der(x, t)
		\\ &= \ip{\tilde u_n}{v}_\fcd^\per - \int_{\R \times \T} h^\per(x) {\tilde u_n}^3 v \der(x, t)
		\\ &\qquad- \sum_{k \in \regularK} \int_\R \frac{\F_k[\calG^\loc(x - x_n)]}{\F_k[\calN]} \bigl(\F_k[\tilde u_n] \overline{\F_k[v]}\bigr) \der x
		- \int_{\R \times \T} h^\loc(x - x_n) {\tilde u_n}^3 v \der(x, t)
		\\ &\to \ip{u}{v}_\fcd^\per - \int_{\R \times \T} h^\per(x) u^3 v \der (x, t)
		= (\fc^\per)'(u)[v]
	\end{align*}
	where we used $\abs{x_n} \to \infty$ and $\calG^\loc(x) \to 0, h^\loc(x) \to 0$ as $x \to \pm\infty$. This shows that $u \neq 0$ is a critical point of $\fc^\per$, and in particular $\fc^\per(u) \geq c_\gs^\per$ holds. However, for fixed $R > 0$ we have
	\begin{align*}
		&\sum_{k \in \regularK} \frac{1}{\omega^2 k^2 \F_k[\calN]} \int_{-R}^R \bigl(\abs{\F_k[u]'}^2 + \omega^2 k^2 V_k^\per(x) \abs{\F_k[u]}^2\bigr) \der x
		\\ &\leq \liminf_{n \to \infty} \sum_{k \in \regularK} \frac{1}{\omega^2 k^2 \F_k[\calN]} \int_{-R}^R \bigl(\abs{\F_k[\tilde u_n]'}^2 + \omega^2 k^2 V_k^\per(x) \abs{\F_k[\tilde u_n]}^2\bigr) \der x
		\\ &= \liminf_{n \to \infty} \sum_{k \in \regularK} \frac{1}{\omega^2 k^2 \F_k[\calN]} \int_{x_n-R}^{x_n+R} \bigl(\abs{\F_k[u_n]'}^2 + \omega^2 k^2 V_k^\per(x) \abs{\F_k[u_n]}^2\bigr) \der x
		\\ &= \liminf_{n \to \infty} \sum_{k \in \regularK} \frac{1}{\omega^2 k^2 \F_k[\calN]} \int_{x_n-R}^{x_n+R} \bigl(\abs{\F_k[u_n]'}^2 + \omega^2 k^2 V_k(x) \abs{\F_k[u_n]}^2\bigr) \der x
		\\ &\leq \liminf_{n \to \infty} \ip{u_n}{u_n}_\fcd
	\end{align*}
	from which $\ip{u}{u}_\fcd^\per \leq \liminf_{n\to\infty} \ip{u_n}{u_n}_\fcd$ follows in the limit $R\to \infty$. This implies
	\begin{align*}
		c 
		&< c_\gs^\per 
		\leq \fc^\per(u) 
		= \fc^\per(u) - \tfrac14 (\fc^\per)'(u)[u] 
		= \tfrac14 \ip{u}{u}_\fcd^\per
		\\ &\leq \tfrac14 \liminf_{n\to\infty} \ip{u_n}{u_n}_\fcd
		= \liminf_{n \to \infty} \fc(u_n) - \tfrac14 \fc'(u_n)[u_n] = c,
	\end{align*}
	a contradiction. 

	
	Thus we have shown the claim. By \cref{lem:Lp_embedding,lem:ps_bounded} up to a subsequence we have $u_n \wto u$ in $\fcd$ and $u_n \to u$ in $L^4_\loc$, where we now know $u \neq 0$. Applying \cref{lem:weakly_convergent_ps} completes the proof.
\end{proof}

\begin{proof}[Proof of \cref{thm:ground_state} for \ref{ass:nonlin:per+loc}]
	Assume first that $c_\gs < c_\gs^\per$ holds. 
	Let $u_n$ be a sequence of critical points of $\fc$ with $\fc(u_n) \to c_\gs$. From \cref{lem:positive_energy,lem:ps_convergence_part3} it follows that there exists a ground state of $\fc$. 
	In the general situation, we distinguish between two cases.
	
	\emph{Case 1:} If there exists a ground state $u^\per$ of $\fc^\per$ which also is a critical point of $\fc$ then clearly $c_\gs \leq c_\gs^\per$ holds. If $c_\gs < c_\gs^\per$ there is nothing left to show, and when $c_\gs = c_\gs^\per$ then  $u^\per$ is a ground state of $\fc$.

	\emph{Case 2:} If no ground state of $\fc^\per$ solves $\fc'(u) = 0$, then by \cref{lem:compare:energy} there exists a Palais-Smale sequence $u_n$ for $\fc$ at some level $c_\mopa \in (0, c_\gs^\per)$. Since $c_\gs\leq c_\mopa$ by Lemma~\ref{lem:ps_convergence_part3}, this shows $c_\gs \leq c_\mopa < c_\gs^\per$.
\end{proof}


\section{Regularity}
\label{sec:regularity}

So far we have shown existence of a ground state to \eqref{eq:effective_problem}. In this section, we discuss its regularity properties.

We proceed in two steps. First, we show regularity for the solution $u$ to \eqref{eq:effective_problem:small}: It is infinitely differentiable in time, twice differentiable in space, and derivatives lie in $L^2 \cap L^\infty$. We also show that if the material parameters are $l$ times continuously differentiable, then $u$ is $l+2$ time differentiable in space and derivatives lie in $L^2 \cap C_b$.

Then we transfer this regularity from the function $u$ to the electromagnetic fields $\bfD, \bfE, \bfB, \bfH$ since these can be expressed as functions of $u$.

We begin by showing infinite time differentiability in the space $\fcd$, see \cref{lem:regularity:time_derivative}, which we prepare with an auxiliary result.

\begin{lemma}\label{lem:fractional_triple_product}
	Let $s > 0$ and $u, \fracDT{s} u \in L^p(\R \times \T)$ where $p \in [3, \infty]$. Then $\fracDT{s}(u^3) \in L^{\nicefrac{p}{3}}(\R \times \T)$.
\end{lemma}
\begin{proof}
	By \cite[Proposition~1]{fractional_leibniz} the estimate
	\begin{align*}
		\norm{\fracDT{s} vw}_{L^r(\T)}
		\lesssim \norm{\fracDT{s} v}_{L^{p_1}(\T)} \norm{w}_{L^{q_1}(\T)}
		+ \norm{v}_{L^{p_2}(\T)} \norm{\fracDT{s} w}_{L^{q_2}(\T)}
	\end{align*}
	holds for all $r, p_j, q_j \in [1, \infty]$ with $\frac{1}{r} = \frac{1}{p_j} + \frac{1}{q_j}$ and $v \in C^\infty(\T)$. By a density argument we obtain 
	\begin{align*}
		\norm{\fracDT{s} (u^3)}_{L^{\nicefrac{p}{3}}(\R \times \T)}
		&= \Norm{\norm{\fracDT{s} (u^3)}_{L^{\nicefrac{p}{3}}(\T)}}_{L^{\nicefrac{p}{3}}(\R)}
		\\ &\lesssim \Norm{
			\norm{\fracDT{s} u}_{L^p(\T)} 
			\norm{u^2}_{L^{\nicefrac{p}{2}}(\T)}
			+ \norm{u}_{L^p(\T)}
			\norm{\fracDT{s} (u^2)}_{L^{\nicefrac{p}{2}}(\T)}
		}_{L^{\nicefrac{p}{3}}(\R)}
		\\ &\lesssim \Norm{
			\norm{\fracDT{s} u}_{L^p(\T)} 
			\norm{u}_{L^p(\T)}^2
		}_{L^{\nicefrac{p}{3}}(\R)}
		\\ &\leq \Norm{\norm{\fracDT{s} u}_{L^{p}(\T)}}_{L^{p}(\R)} \Norm{\norm{u}_{L^{p}(\T)}}_{L^{p}(\R)}^2.
		\qedhere
	\end{align*}
\end{proof}

\begin{lemma}\label{lem:regularity:time_derivative}
	Let $u \in \fcd$ be a critical point of $\fc$. Then $\fracDT{s} u \in \fcd$ for all $s \in \R$.
\end{lemma}
\begin{proof}
	Since $u \in \fcd$, $\fracDT{s} u \in \fcd$ holds for $s \leq 0$. Moreover, if $\fracDT{s} u \in \fcd$ then $\fracDT{\sigma} u \in \fcd$ for all $\sigma\leq s$. By \cref{cor:Lp_embedding:derivatives} there exists $\eps > 0$ such that $\fracDT{\eps} \colon \fcd \to L^4(\R \times \T)$ is bounded.
	We show by induction that $\fracDT{n \eps} u \in \fcd$ holds for $n \in \N_0$. So assume $\fracDT{n \eps} u \in \fcd$ for fixed $n \in \N_0$. 
	Let $v \in \fcd$ with $\fracDT{(n+1)\eps} v \in \fcd$. Then we have
	\begin{align*}
		0 &= \fc'(u)[\fracDT{(n+1)\eps} v]
		\\ &= \ip{u}{\fracDT{(n+1)\eps} v}_\fcd - \int_{\R \times \T} h(x) u^3 \cdot \fracDT{(n+1)\eps} v \der (x, t)
		\\ &= \ip{\fracDT{n\eps}u}{\fracDT{\eps} v}_\fcd - \int_{\R \times \T} h(x) \fracDT{n \eps}(u^3) \cdot \fracDT{\eps} v \der (x, t).
	\end{align*}
	By \cref{lem:fractional_triple_product} with $p=4$ and a density argument (cf. \cref{lem:dense_subset}) we see that the map $v \mapsto \int_{\R \times \T} h(x) \fracDT{n \eps}(u^3) \cdot \fracDT{\eps} v \der (x, t)$ extends to a bounded linear functional on $\fcd$. Hence there exists $w \in \fcd$ with
	\begin{align*}
		\ip{w}{v}_\fcd
		= \int_{\R \times \T} h(x) \fracDT{n \eps}(u^3) \cdot \fracDT{\eps} v \der (x, t) 
		= \ip{\fracDT{n\eps}u}{\fracDT{\eps} v}_\fcd
	\end{align*}
	for $v \in \fcd$ with $\fracDT{(n+1)\eps}v \in \fcd$. Again by density we get $\fracDT{(n+1)\eps} u = w$.
\end{proof}

In order to proceed we need the following little result on the mapping properties of Fourier multiplier operators.
\begin{lemma}\label{lem:multiplier_invariance}
	Let $M v = \calF^{-1}[m_k \hat v_k]$ be a Fourier multiplier with symbol $\abs{m_k} \lesssim \abs{k}^\sigma$ of polynomial growth and let $u$ be a function with $\fracDT{s} u \in \fcd$ for all $s \in \R$. Then $\fracDT{s} M u \in \fcd$ for all $s \in \R$. The same holds for $\fcd$ replaced by $L^p(\R \times \T)$ with $p \in [1, \infty]$ if we require $\hat u_0=0$.
\end{lemma}
\begin{proof}
    In the Hilbert space setting we have
	\begin{align*}
		\norm{\fracDT{s} M u}_\fcd \leq (\sup_{k \in \regularK} \abs{m_k} \abs{\omega k}^{-\sigma}) \norm{\fracDT{s + \sigma} u}_\fcd < \infty
	\end{align*}
	In the $L^p(\R \times \T)$ case, the series $\mu(t) = \sum_{k \in \Z\setminus\{0\}} m_k \abs{\omega k}^{-\sigma-1} e_k(t)$ converges in $L^1(\T)$. Hence
	\begin{align*}
		\norm{\fracDT{s} M u}_p 
		= \norm{\mu \ast \fracDT{s+\sigma+1} u}_p
		\leq \norm{\mu}_1 \norm{\fracDT{s+\sigma+1} u}_p 
		< \infty.
		&\qedhere
	\end{align*}
\end{proof}
Note that \cref{lem:multiplier_invariance} applies to the multipliers $\calN \ast, \calG(x) \ast$ and by \ref{ass:nonlin:decay} also to $(\calN\ast)^{-1}$.

Continuing our regularity analysis we show that $u$ and its derivatives lie in $L^2 \cap L^\infty$. This also shows that $u$ satisfies \eqref{eq:effective_problem} strongly.

\begin{proposition}\label{prop:regularity:W2p}
	Let $u \in \fcd$ be a critical point of $\fc$. Then $\fracDT{s} u \in W^{2,p}(\R \times \T)$ for all $s \in \R$ and $p \in [2, \infty]$, and it satisfies the equation
	\begin{align}\label{eq:differentiated_u_problem}
		- u_{xx} - V(x) \partial_t^2 u + h(x) \partial_t^2\left(\calN \ast u^3\right) = 0.
	\end{align}
	If \ref{ass:regularity} holds, we moreover have $\fracDT{s} u \in W^{2+l,2}(\R \times \T) \cap C_b^{2+l}(\R \times \T)$. 
\end{proposition}
\begin{proof}
	We remark that equation \eqref{eq:differentiated_u_problem} formally follows by applying $-\partial_t^2 \calN \ast$ to \eqref{eq:effective_problem}.

	\textit{Part 1:} We first show $\fracDT{s} u \in L^p(\R \times \T)$.
	Because of boundedness of the embedding $\fcd \embeds L^2(\R\times\T)$ and interpolation, it suffices to give the result for $p = \infty$. Similarly as in the proof of \cref{lem:Lp_embedding} we calculate
	\begin{align*}
		\norm{\fracDT{s} u}_\infty
		\lesssim \norm{\abs{\omega k}^s \F_{\xi,k}[u]}_{1}
		\leq \norm{\abs{\omega k}^{-\nicefrac12} \sqrt{\frac{\omega^2 k^2 \F_k[\calN]}{\xi^2 + \omega^2 k^2}}}_2 
		\cdot \norm{\abs{\omega k}^{s + \nicefrac12} \sqrt{\frac{\xi^2 + \omega^2 k^2}{\omega^2 k^2 \F_k[\calN]}} \F_{\xi, k}[u]}_2
	\end{align*}
	where the first term is finite since $0 \leq \F_k[\calN] \leq \abs{k}^{-\alpha}$, $\alpha > 1$, and the second term is equivalent to $\norm{\fracDT{s + \nicefrac12} u}_\fcd$ and thus finite by \cref{lem:regularity:time_derivative}.

	\textit{Part 2:}
	Next we show $\fracDT{s} u_{x} \in L^2(\R \times \T)$:
	\begin{align*}
		\norm{\fracDT{s} u_x}_2^2
		&= \sum_{k \in \regularK} \int_\R \abs{\omega k}^{2s} \abs{\hat u_k'}^2 \der x
		\\ &\lesssim \sum_{k \in \regularK} \int_\R \frac{\abs{\omega k}^{2s + 2 - \alpha}}{\omega^2 k^2 \F_k[\calN]} \abs{\hat u_k'}^2 \der x
		\leq \norm{\fracDT{s + 1 - \nicefrac{\alpha}{2}} u}_\fcd^2 
		< \infty.
	\end{align*}

	Now let $v \in \fcd$ with $\fracDT{s+2} \calN \ast v \in \fcd$. As $\calN$ is even by \ref{ass:nonlin:decay}, we have
	\begin{align*}
		0 
		&= \fc'(u)[\fracDT{s+2} \calN \ast v]
		\\ &= \sum_{k \in \regularK} \abs{\omega k}^s \int_\R \hat u_k' \overline{\hat v_k'} + \omega^2 k^2 V_k(x) \hat u_k \overline{\hat v_k} \der x
		- \int_{\R \times \T} h(x) u^3 \cdot \fracDT{s+2} \calN \ast v \der (x, t)
		\\ &= \int_{\R \times \T} \fracDT{s} u_x \cdot v_x \der (x, t)
		+ \int_{\R \times \T} \fracDT{2+s}\left( V(x) u - h(x) \calN \ast u^3 \right) \cdot v \der (x, t).
	\end{align*}
	Since $v$ was arbitrary, by density it follows that 
	\begin{align}\label{eq:loc:fc_prime_pde}
		\fracDT{s} u_{xx} = \fracDT{2+s}\left( V(x) u - h(x) \projR[\calN \ast u^3] \right)
	\end{align}
	holds. 
	The term on the right-hand side lies in $L^p(\R \times \T)$ by  \cref{lem:fractional_triple_product,lem:multiplier_invariance} and the first part of the proof. Thus, $\fracDT{s} u_{xx} \in L^p(\R \times \T)$ and $\fracDT{s} u \in W^{2, p}(\R \times \T)$.

	\textit{Part 3:} Assume \ref{ass:regularity}, i.e. $\calG \in C_b^l(\R; \calM(\T))$, $h \in C_b^l(\R)$.
	First, we have $\fracDT{s} u \in C_b(\R \times \T)$ by Part~1 and Sobolev's embedding. Continuity of $\calG, h$ shows that the right-hand side of \eqref{eq:loc:fc_prime_pde} is continuous, so $\fracDT{s} u_{xx} \in C_b(\R \times \T)$ holds, and in particular $\fracDT{s} u \in C_b^2(\R \times \T)$. 

	For $l > 0$ we argue by induction over $k=0,\ldots, l$. We use that by \eqref{eq:loc:fc_prime_pde} we have
	\begin{align*}
		\partial_x^{k+2} \fracDT{s} u = \partial_x^k \fracDT{s+2} \left(V(x) u - h(x) \calN \ast u^3\right)
	\end{align*}
	where the right-hand side lies in $L^2(\R \times \T) \cap C_b(\R \times \T)$ by the product rule and the induction hypothesis. This allows us to conclude $\fracDT{s} u \in W^{2+k,2}(\R \times \T) \cap C_b^{k+2}(\R \times \T)$.
\end{proof}

Recall for the first type of nonlinearity \eqref{eq:material:noni} that the profile $w$ of the electric field is given by $w = u$. For the second type of nonlinearity \eqref{eq:material:nonii}, by \eqref{eq:w_reconstruction:ii} the profile satisfies $\projR[w] = (\calN\ast)^{-1} u$ where $\projS[w]$ solves a differential equation. Therefore we need to discuss next the regularity of $w$ for the second type of nonlinearity, which is done in the following analogue to \cref{prop:regularity:W2p}.

\begin{proposition}\label{prop:regularity:w}	
	Let $u \in \fcd$ be a critical point of $\fc$. Define $w = w_1 + w_2$ where
	\begin{align*}
		w_1 = \projR[w] = (\calN\ast)^{-1} u,
		\quad
		w_2 = \projS[w] = (-\partial_x^2 - V(x) \partial_t^2)^{-1} (h(x) \partial_t^2 \projS[u^3]).
	\end{align*}
	Then $w$ satisfies $\fracDT{s} w \in W^{2,p}(\R \times \T)$ for all $s \in \R$, $p \in [2, \infty]$ and solves
	\begin{align*}
		(- \partial_x^2 - V(x) \partial_t^2) w + h(x) \partial_t^2 (\calN \ast w)^3 = 0.
	\end{align*}
	If \ref{ass:regularity} holds, we moreover have $\fracDT{s} w \in W^{2+l,2}(\R \times \T) \cap C_b^{2+l}(\R \times \T)$.
\end{proposition}
\begin{proof}
	First, by \cref{lem:multiplier_invariance,prop:regularity:W2p} we see that the function $w_1 \coloneqq (\calN\ast)^{-1} u \coloneqq \sum_{k \in \regularK} \frac{1}{\F_k[\calN]} \hat u_k(x) e_k(t)$ satisfies $\fracDT{s} w_1 \in W^{2,p}(\R \times \T)$ for $s \in \R, p \in [2, \infty]$, with additional regularity if $\calG, h$  fulfills \ref{ass:regularity}. Applying $(\calN\ast)^{-1}$ to \eqref{eq:differentiated_u_problem} we see that $w_1$ solves
	\begin{align*}
		(- \partial_x^2 - V(x) \partial_t^2) w_1 + h(x) \partial_t^2 \projR[u^3] = 0.
	\end{align*}
	Let us now turn our attention to $w_2$ and define the space
	\begin{align*}
		H_2 \coloneqq \set{v \in H^1(\R \times \T) \colon \hat v_k \equiv 0 \text{ for } k \in \regularK \cup \set{0}}.
	\end{align*}
	Since $V_k(x)$ is bounded and positive, the Riesz representation theorem provides $w_2 \in H_2$ with
	\begin{align}\label{eq:loc:projected_nonlinearity:weak}
		\int_{\R \times \T} \partial_x w_2 \cdot \partial_x v + V(x) \partial_t w_2 \cdot \partial_t v \der (x, t)
		= - \int_{\R \times \T} h(x) \partial_t^2 \projS[u^3] \cdot v \der (x, t)
	\end{align}
	for all $v \in H_2$. 

	Similarly as for solutions $u$ of $\fc'(u)[v] = 0$ for $v \in \fcd$, we obtain regularity for the solution $w_2$ to \eqref{eq:loc:projected_nonlinearity:weak}. In contrast to $\fc$, where critical points satisfy a truly nonlinear equation, the right-hand side of \eqref{eq:loc:projected_nonlinearity:weak} is independent of $w_2$ and its regularity properties have been established in \cref{prop:regularity:W2p}. Let us sketch the arguments: 

	As in \cref{lem:regularity:time_derivative} we find $\fracDT{s} w_2 \in H_2$ for $s \in \R$. Using that the $0$-th Fourier mode of $h(x) \partial_t^2 \projS[u^3]$ vanishes, $w_2$ satisfies 
	\begin{align}\label{eq:loc:projected_nonlinearity:pw}
		- \partial_x^2 \fracDT{s} w_2 - V(x) \partial_t^2 \fracDT{s} w_2 = - h(x)\partial_t^2 \fracDT{s} \projS[u^3]
	\end{align}
	By the fractional Leibniz rule from \cref{lem:fractional_triple_product}, the regularity properties of $u$ from \cref{prop:regularity:W2p} and the boundedness of the Fourier symbol of $\projS$ we find that the right-hand side of \eqref{eq:loc:projected_nonlinearity:pw} lies in $L^2$. Therefore $\fracDT{s} w_2 \in W^{2,2}(\R \times \T) \subseteq L^\infty(\R \times \T)$ holds for $s \in \R$ and \eqref{eq:loc:projected_nonlinearity:pw} shows $\fracDT{s} w_2 \in W^{2,p}(\R \times \T)$ for $s \in \R$, $p \in [2, \infty]$.
	The additional regularity when $\calG, h$ satisfy \ref{ass:regularity} can then be shown as in \cref{prop:regularity:W2p} by iteratively applying space-derivatives to \eqref{eq:loc:projected_nonlinearity:pw}.
\end{proof}

Lastly, we discuss the regularity of the corresponding electromagnetic fields.

\begin{proof}[Proof of \cref{thm:main} for slab geometries]
	Let $u$ be a nontrivial critical point of $\fc$. If the nonlinearity is given by $N(w) = \calN \ast w^3$ then we set $w:= u$. If otherwise $N(w) = (\calN \ast w)^3$ then let $w$ be from \cref{prop:regularity:w}. Then we define $W \coloneqq \partial_t^{-1} w$ and reconstruct the electromagnetic fields by
	\begin{align*}
		\bfE(\bfx, t) &= w(x, t - \tfrac{1}{c} z) \cdot \begin{pmatrix}
			0 \\ 1 \\ 0
		\end{pmatrix},
		&\bfB(\bfx, t) &= - \begin{pmatrix}
			\tfrac{1}{c} w(x, t - \tfrac{1}{c} z) \\ 0 \\ W_x(x, t - \tfrac{1}{c} z)
		\end{pmatrix} 
		\\ \bfD(\bfx, t) &= \epsilon_0 \left( w + \calG \ast w + h(x) N(w) \right) \cdot \begin{pmatrix}
			0 \\ 1 \\ 0
		\end{pmatrix}, 
		&\bfH(\bfx, t) &= \tfrac{1}{\mu_0} \bfB(\bfx, t).
	\end{align*} 
	Due to \cref{prop:regularity:W2p} and \cref{prop:regularity:w} we have the inclusions
	\begin{align*}
		\partial_t^n \bfE \in W^{2,p}(\Omega; \R^3),
		\qquad 	
		\partial_t^n \bfB, \partial_t^n \bfH \in W^{1,p}(\Omega; \R^3),
		\qquad
		\partial_t^n \bfD \in L^{p}(\Omega; \R^3),
	\end{align*}
	and assuming \ref{ass:regularity} we moreover have
	\begin{align*}
		\partial_t^n \bfE \in \tilde C_b^{2+l}(\Omega; \R^3),
		\qquad  
		\partial_t^n \bfB, \partial_t^n \bfH \in \tilde C_b^{1+l}(\Omega; \R^3),
		\qquad
		\partial_t^n \bfD \in \tilde C_b^{l}(\Omega; \R^3)
	\end{align*}
	for any domain $\Omega = \R \times [y, y+1] \times [z, z + 1] \times [t, t+1]$ and all $n \in \N$, $p \in [2, \infty]$ with norm bounds independent of $y, z, t$. By direct calculation one checks that the fields $\bfE, \bfD, \bfB, \bfH$ solve Maxwell's equations \eqref{eq:maxwell}, \eqref{eq:material}.
\end{proof}

\begin{proof}[Proof of \cref{rem:infinitely}]
	To show that there exist infinitely many solutions, we search for breather solutions with time period $\frac{T}{n}$ instead of $T$. If we define the corresponding time-domain $\T_n \coloneqq \R /_{\frac{T}{n} \Z}$ then the $k$-th Fourier coefficient of a $\frac{T}{n}$-periodic function $f$, when understood as a $\T$-periodic function, satisfies $\F_k[f \colon \T_n \to \R] = \F_{nk}[f \colon \T \to \R]$. Therefore, by the above arguments there exists a $\frac{T}{n}$-periodic breather solution $(\bfE_n, \bfD_n, \bfB_n, \bfH_n)$ to \eqref{eq:maxwell}, \eqref{eq:material} with minimal period $T_n > 0$ for all such $n \in \N$ where $\regularK \cap n \Z \neq \emptyset$. By assumption, the set $\set{n \in \N \colon \regularK \cap n \Z \neq \emptyset}$ is infinite. Since $T_n$ is a divisor of $\frac{T}{n}$ the minimal period $T_n$ goes to $0$ for $n \to \infty$. Hence infinitely many among the breather solutions $(\bfE_n, \bfD_n, \bfB_n, \bfH_n)$ must be mutually distinct.
\end{proof}


\section{Modifications for the cylindrical geometry}
\label{sec:cylinder_modifications}

In this section we discuss the cylindrical problem, that is, we consider equation \eqref{eq:profile_problem:cyl} instead of \eqref{eq:profile_problem:slab}. 
The only difference between the two problems is in the spatial differential operator, where we now work with $-\partial_r^2 - \tfrac{1}{r} \partial_r + \tfrac{1}{r^2}$ on the domain $r \in [0, \infty)$ instead of $-\partial_x^2$ for $x \in \R$. 
The differential operator $- \partial_r^2 - \tfrac{1}{r} \partial_r$ is the $2$d Laplacian for radially symmetric functions, and $\tfrac{1}{r^2}$ is an additional positive term. Hence it is natural to equip the domain $[0, \infty)$ with the measure $r \Der r$, and to identify functions on it with radially symmetric functions of the variables $(x, y) \in \R^2$ via $r = \sqrt{x^2 + y^2}$. We use the subscript ``$\rad$'' to denote spaces of functions that are radially symmetric in $(x,y)$. Since the term $\int_0^\infty \frac{u^2}{r^2} \der[r] r$ cannot be controlled by the $H^1_\rad$-Sobolev norm of $u$ (recall that Hardy's inequality fails in two dimensions) we need to add this term in the form domain of the differential operator.

We will discuss how the arguments from \cref{sec:variational_problem,sec:ground_state_existance,sec:regularity} have to be adapted to treat the cylindrical problem. We use the same structure as in these sections. 
In order to not repeat the previous chapters, we discuss in detail only results that require new techniques to adapt them to the cylindrical geometry and roughly sketch the other results.

\subsection{Modifications for Sections~\ref{sec:variational_problem}~and~\ref{sec:ground_state_existance}}

In analogy to \cref{def:slab:functional_domain,def:slab:weak_solution}, we define the functional of interest $\fcr$ and its domain $\fcrd$ (replacing $\fc$ and $\fcd$).
\begin{definition}
	We define the space
	\begin{align*}
		\fcrd \coloneqq \set{u \in L^2([0, \infty) \times \T; r \Der (r, t)) \colon \hat u_k = 0 \text{ for } k \in \Zeven \cup \singularK, \norm{u}_\fcrd^2 \coloneqq \iip{u}{u}_\fcrd < \infty}
	\end{align*}
	with the two equivalent inner products
	\begin{align*}
		\iip{u}{v}_\fcrd &\coloneqq \sum_{k \in \regularK} \frac{1}{\omega^2 k^2 \F_k[\calN]} 
		\int_0^\infty \left(\hat u_k' \overline{\hat v_k'} + \bigl(\tfrac{1}{r^2} + \omega^2 k^2\bigr) \hat u_k \overline{\hat v_k} \right) \der[r] r,
		\\ \ip{u}{v}_\fcrd &\coloneqq \sum_{k \in \regularK} \frac{1}{\omega^2 k^2 \F_k[\calN]} \int_0^\infty \bigl(\hat u_k' \overline{\hat v_k'} + \bigl(\tfrac{1}{r^2} + \omega^2 k^2 \tilde V_k(r)\bigr) \hat u_k \overline{\hat v_k}\bigr) \der[r]r
	\end{align*}
	where $\tilde V_k(r) \coloneqq \tfrac{1}{c^2} - 1 - \F_k[\calG(r)]$. On $\fcrd$, we define the functional 
	\begin{align*}
		\fcr(u) \coloneqq \tfrac12\ip{u}{u}_\fcrd - \tfrac14 \int_{[0, \infty) \times \T} h(x) u^4 \der[r]{(r,t)}
		\quad\text{for}\quad
		u \in \fcrd
	\end{align*}
	so that its critical points $u \in \fcrd$ satisfy
	\begin{align*}
		\fcr'(u)[v] 
		= \ip{u}{v}_\fcrd - \int_{[0, \infty) \times \T} h(x) u^3 v \der[r]{(r,t)} = 0
		\quad\text{for}\quad
		v \in \fcrd
	\end{align*}
\end{definition}

As a first step, we discuss the embedding properties of $\fcrd$.

\begin{lemma}\label{lem:cyl:Lp_embedding}
	For any $p \in (2, p^\star)$ with $p^\star = \frac{6}{3 - \alpha}$ ($p^\star = \infty$ if $\alpha \geq 3$), the embedding $\fcrd \embeds L^p([0, \infty) \times \T; r \Der (r,t))$ is compact.
	Moreover, $\fcrd \embeds L^2([0, \infty) \times \T; r \Der (r,t))$ is continuous and $\fcrd \embeds L^2_\loc([0, \infty) \times \T; r \Der (r,t))$ is compact.
\end{lemma}
\begin{proof}
	We interpret a function $u \colon [0, \infty) \times \T \to \R$ as a function of the three variables $(x,y,t)$ which is radially symmetric in $(x,y)$ via $r = \sqrt{x^2 + y^2}$. Let $\sigma_\rho$ be the surface measure of the sphere $S_\rho \subseteq \R^2$ of radius $\rho$ centered at $0$, normalized such that $\sigma_\rho(S_\rho) = 1$, and continued by $0$ to a Borel measure on $\R^2$. It satisfies
	\begin{align*}
		\F^{-1}_{(x,y)}(\sigma_\rho) = \tfrac{1}{2 \pi} J_0(r \rho)
	\end{align*}
	where $J_0$ is the Bessel function of first kind.
	Using $\abs{J_0(s)} \lesssim s^{-\theta}$ for $\theta \in [0, \tfrac12]$ (cf. \cite{gradshteyn}) we obtain
	\begin{align*}
		\norm{r^\theta u}_{L^\infty(\R^2 \times \T)}
		\lesssim \norm{\abs{\xi}^{-\theta} \F_{\xi,k}[u]}_{L^1(\R^2 \times \Z)},
	\end{align*}
	where we used that $\F_{\xi,k}[u]$ is radially symmetric in $\xi$.
	Note also that we have $\norm{u}_{L^2(\R^2 \times \T)} = \norm{\F_{\xi,k}[u]}_{L^2(\R^2 \times \Z)}$. This allows us to use the Riesz-Thorin interpolation theorem (cf. \cite{lunardi}) and get (with $\Der\#$ denoting the counting measure) that the map 
	$$
	T: \left\{ \begin{array}{rcl} 
	L_\rad^{p'}(\R^2\times\Z; |\xi|^{-2\theta}\Der \xi\otimes \Der\#) & \to & L^p_\rad(\R^2\times\T; r^{-2\theta}\Der (x,t)), \vspace{\jot} \\
	v & \mapsto & r^\theta \F_{\xi,k} (|\xi|^{-\theta} v)
	\end{array} \right.
	$$
	is bounded for all $p\in [2,\infty]$. To see this note that with $v =|\xi|^{-\theta} \F_{\xi,k}[u]$ we have 
	\begin{align*} 
	\|v\|_{L^{p'}(\R^2\times\Z; |\xi|^{-2\theta}\Der \xi\otimes\Der\#)} &= \|\F_{\xi,k}[u] |\xi|^{-\theta_0}\|_{L^{p'}(\R^2\times\Z; \Der \xi\otimes\Der\#)} \\
	\|Tv\|_{L^p(\R^2\times\T; r^{-2\theta}\Der (x,t))} &= \|r^{\theta_0} u\|_{L^p(\R^2\times\T; \Der (x,t))}
	\end{align*}
	where $\theta_0 = \theta \frac{(p-2)}{p}$ ranges through $[0,\tfrac12- \tfrac1p]$ as $\theta$ runs through $[0,\tfrac12]$. 
	Thus we have 
	\begin{align*}
	\norm{r^{\theta_0} u}_{L^p(\R^2 \times \T)}
	&\lesssim \norm{\abs{\xi}^{-\theta_0} \F_{\xi, k}[u]}_{L^{p'}(\R^2 \times \regularK)}
	\\ &\leq \norm{\abs{\xi}^{-\theta_0} \sqrt{\frac{\omega^2 k^2 \F_k[\calN]}{\abs{\xi}^2 + \omega^2 k^2}}}_{L^r(\R^2 \times \regularK)}
	\norm{\sqrt{\frac{\abs{\xi}^2 + \omega^2 k^2}{\omega^2 k^2 \F_k[\calN]}} \F_{\xi,k}[u]}_{L^2(\R^2 \times \regularK)}
	\end{align*}
	where $\frac{1}{r} = \frac{1}{2} - \frac{1}{p} < \frac{\alpha}{6}$. 
	By the choice of $p^\star$ and assumption \ref{ass:nonlin:decay}, the $L^r$-norm is finite provided $\theta_0$ is chosen sufficiently small, and the $L^2$-norm can be estimated against $\norm{u}_\fcrd$.
	
	For the particular choice $\theta_0 = 0$, this shows that the embedding $\fcrd \embeds L^p_\rad(\R^2 \times \T)$ is continuous.
	Moreover, we can argue similarly as in the proof of \cref{lem:Lp_embedding} to verify that the local embedding $\fcrd \embeds L^p_{\rad,\loc}(\R^2 \times \T)$ is compact.
	
	It remains to show that $\fcrd \embeds L^p_\rad(\R^2 \times \T)$ is compact for $p \neq 2$. 
	For $R > 0$ consider the compact map $E_R \colon \fcrd \to L^p_\rad(\R^2 \times \T), u \mapsto u \bbone_{B_R(0) \times \T}$. 
	Using the above inequality we have 
	\begin{align*}
		\norm{E_R u - u}_p
		\leq \norm{\left(\tfrac{r}{R}\right)^{\theta_0} u}_p
		\lesssim R^{-\theta_0} \norm{u}_\fcrd
	\end{align*}
	so by choosing any admissible $\theta_0 > 0$ and taking the limit $R \to \infty$ we see that the embedding $I \colon \fcrd \to L^p_\rad(\R^2 \times \T)$ is compact as the uniform limit of a sequence of compact operators.
\end{proof}

Notice that, unlike in the slab setting (cf. \cref{lem:Lp_embedding}), the embedding of \cref{lem:cyl:Lp_embedding} is compact for $p > 2$. This is why we do not require additional assumptions \ref{ass:nonlin:loc} or \ref{ass:nonlin:per+loc} in the cylindrical setting. 
One can then show existence of ground states similar to the ``compact'' case \ref{ass:nonlin:loc} of \cref{sec:ground_state_existance}. The only difference is that existence of a convergent subsequence of $h(r) u_n^3$ in $L^{4/3}_\rad(\R^2\times\T)$ is guaranteed by the compact embedding instead of decay properties of $h$.

\subsection{Modifications for Section~\ref{sec:regularity}}

In the following, we show in \cref{prop:regularity:cyl:W2p,prop:regularity:cyl:w} two regularity results that are the cylindrical counterparts to \cref{prop:regularity:W2p,prop:regularity:w}. 

Here, arguments will get more difficult since the cylindrical geometry is effectively $2$-dimensional in space (compared to $1$d for the slab problem).
For some arguments it will be advantageous to view the $\tfrac{1}{r^2}$ not as an additional order $0$ term, but as part of the differential operator. From \cite{bartsch_et_al} we use the identity
\begin{align}\label{eq:4dlaplace}
	\partial_r^2 + \tfrac{1}{r} \partial_r - \tfrac{1}{r^2}
	= \tfrac{1}{r^2} \partial_r r^3 \partial_r \tfrac{1}{r},
\end{align}
which means that up to the multiplicative factors $r, \tfrac{1}{r}$ we are dealing with $\tfrac{1}{r^3} \partial_r r^3 \partial_r$, which is the Laplacian of a radially symmetric function in $4$ dimensions.

Similar to \cref{prop:regularity:W2p} we show that $u$ and its derivatives lie in $L^2 \cap L^\infty$.

\begin{proposition}\label{prop:regularity:cyl:W2p}
	Let $u \in \fcrd$ be a critical point of $\fcr$. Then the terms 
	\begin{align*}
		\max\set{r, 1} \fracDT{s} (\tfrac{u}{r}),
		\qquad
		\max\set{r, 1} \fracDT{s} \partial_r (\tfrac{u}{r}),
		\qquad 
		r \fracDT{s} \partial_r^2 (\tfrac{u}{r})
	\end{align*}
	lie in $L^{p}([0, \infty) \times \T; r \Der (r,t))$ for all $s \in \R$ and $p \in [2, \infty]$, and $u$ solves pointwise 
	\begin{align*}
		(-\partial_r^2 - \tfrac1r \partial_r + \tfrac{1}{r^2} - V(x) \partial_t^2) u + h(x) \partial_t^2 (\calN \ast u^3) = 0.
	\end{align*}
	If \ref{ass:regularity} holds, then the terms 
	\begin{align*}
		\max\set{r, 1} \fracDT{s} \partial_r^n (\tfrac{u}{r}) 
		\quad\text{for}\quad 0 \leq n \leq l+1
		\qquad \text{as well as} \qquad 	
		r \fracDT{s} \partial_r^{l+2} (\tfrac{u}{r})
	\end{align*}
	lie in $L^2([0, \infty) \times \T; r \Der (r,t)) \cap C_b([0, \infty) \times \T)$. 
	Moreover, the second term vanishes at $r = 0$, and the same holds for the first term when $n$ is odd.
\end{proposition}
\begin{proof}
	\textit{Part 1:} First, following the proof of \cref{lem:regularity:time_derivative} we obtain $\fracDT{s} u \in \fcrd$ for all $s \in \R$.

	Next, for $p \in [2, \infty)$ we calculate 
	\begin{align*}
		\MoveEqLeft \norm{\fracDT{s} u}_p 
		\lesssim \norm{\abs{\omega k}^s \F_{\xi,k}}_{L^{p'}(\R^2 \times \regularK)}
		\\ &\leq \norm{\abs{\omega k}^{-1} \sqrt{\frac{\omega^2 k^2 \F_k[\calN]}{\abs{\xi}^2 + \omega^2 k^2}}}_{L^r(\R^2 \times \regularK)}
		\norm{\abs{\omega k}^{s + 1} \sqrt{\frac{\abs{\xi}^2 + \omega^2 k^2}{\omega^2 k^2 \F_k[\calN]}} \F_{\xi,k}[u]}_{L^2(\R^2 \times \regularK)}
		\lesssim \norm{\fracDT{s + 1} u}_\fcrd.
	\end{align*}
	Here $1 - \frac{1}{p} = \frac{1}{p'} = \frac{1}{r} + \frac{1}{2}$, and the $L^r$-norm is finite since $r > 2$ and therefore
	\begin{align*}
		\MoveEqLeft \sum_{k \in \R} \abs{\omega k}^{- r} \int_{\R^2} \left( \frac{\omega^2 k^2 \F_k[\calN]}{\abs{\xi}^2 + \omega^2 k^2} \right)^{\frac r2} \der \xi
		\\ &= \int_{\R^2} \left(\frac{1}{\abs{\xi}^2 + 1} \right)^{\nicefrac{r}{2}} \der \xi \cdot \sum_{k \in \R} \abs{\omega k}^{2-r} \F_k[\calN]^{\frac r2}
		\lesssim \sum_{k \in \regularK} \abs{k}^{2 - \frac74 r} < \infty.
	\end{align*}

	\textit{Part 2:}
	Arguing as in part 2 of the proof of \cref{prop:regularity:W2p} we have that $\fracDT{s} u_r, \tfrac{1}{r} \fracDT{s} u \in L^2([0, \infty) \times \T; r \Der (r,t))$ for $s \in \R$ and
	\begin{align}\label{eq:loc:fcr_prime_pde:pw}
		\fracDT{s} u_{rr} + \tfrac{1}{r} \fracDT{s} u_r - \tfrac{1}{r^2} \fracDT{s} u
		= \fracDT{s+2} \left(V(r)u - h(r) \projR[\calN \ast u^3]\right)
	\end{align}
	holds pointwise. From now on arguments differ depending on if $r$ is large or small, and we discuss these cases in part~3 and part~4, respectively.

	\textit{Part 3a:} Let $r > R_1$ for fixed $R_1 > 0$. Then part~1 combined with \eqref{eq:loc:fcr_prime_pde:pw} shows that $\Delta_r \fracDT{s} u \coloneqq (\partial_r^2 + \tfrac{1}{r} \partial_r) \fracDT{s} u \in L^p([R_1, \infty) \times \T; r \Der (r,t))$ for $p \in [2, \infty)$.
	Now choose a cutoff $\psi_1 \in C^\infty([0, \infty))$ with $\supp \psi_1 \subseteq (R_1, \infty)$ and $\psi_1 \equiv 1$ on $[R_2, \infty)$ for $R_2 > R_1$. Then, interpreting $v_1 \coloneqq \psi_1(r) u$ as a function of the three variables $(x, y, t)$ via $r = \sqrt{x^2 + y^2}$ and continuing by zero, we have $\fracDT{s} v_1 \in L^p_\rad(\R^2 \times \T)$ and
	\begin{align*}
		\Delta_{(x,y)} \fracDT{s} v_1 = \Delta_r \psi_1 \cdot \fracDT{s} u + 2 \partial_r \psi_1 \cdot \partial_r \fracDT{s} u + \psi_1 \cdot \Delta_r \fracDT{s} u \in L^2(\R^2 \times \T).
	\end{align*}
	This shows $\fracDT{s} v_1 \in H^2_\rad(\R^2 \times \T)$, and by Sobolev's embedding we in particular have $\fracDT{s} u \in L^\infty([R_2, \infty) \times \T)$, $\fracDT{s} u_r \in L^6([R_2, \infty) \times \T; r \Der (r,t))$. 

	Similar to above, but now with a smooth cutoff $\psi_2$ such that $\supp \psi_2 \subseteq (R_2, \infty)$, $\psi_2 \equiv 1$ on $[R_3, \infty)$ for $R_3 > R_2$, we see that $v_2 \coloneqq \psi_2(r) u$ satisfies $\Delta_{(x,y)} \fracDT{s} v_2 \in L^6_\rad(\R^2 \times \T)$ where again \eqref{eq:loc:fcr_prime_pde:pw} was used. Thus $\fracDT{s} v_2 \in W^{2,6}_\rad(\R^2 \times \T)$ by $L^p$-boundedness of the Riesz transform, cf. \cite[Corollary~5.2.8]{grafakos_classical}. By Sobolev's embedding we have $ \partial_r \fracDT{s} u \in L^\infty([R_3, \infty) \times \T)$, and then $\Delta_r \fracDT{s} u \in L^\infty([R_3, \infty) \times \T)$ by \eqref{eq:loc:fcr_prime_pde:pw}. 
	So far we have shown
	\begin{align*}
		\fracDT{s} u, \fracDT{s} u_r, \fracDT{s} u_{rr} 
		\in L^2([R_3, \infty) \times \T; r \Der (r,t)) \cap L^\infty([R_3, \infty) \times \T).
	\end{align*}
	This shows the first part of \cref{prop:regularity:cyl:W2p} for $r > R_3$, where $R_3 > 0$ can be chosen arbitrarily.

	\textit{Part 3b:} We assume \ref{ass:regularity}, i.e. $\calG, h \in C_b^l$, and still consider large $r$. From \textit{Part 3a} we obtain $\fracDT{s} u, \fracDT{s} u_r, \in C_b([R, \infty) \times \T)$ from the high Sobolev regularity and therefore also $\fracDT{s} u_{rr}\in C_b([R, \infty) \times \T)$ by applying \eqref{eq:loc:fcr_prime_pde:pw}. 
	Now $\fracDT{s} \partial_r^n u \in L^2([R, \infty) \times \T; r \Der (r,t)) \cap C_b([R, \infty) \times \T)$ for $2 < n \leq l + 2$ 
	can be shown iteratively by applying space-derivatives to \eqref{eq:loc:fcr_prime_pde:pw} and using that all terms except the highest order space-derivative term lie in $L^2([R, \infty) \times \T; r \Der (r,t)) \cap C_b([R, \infty) \times \T)$ by the induction hypothesis.

	\textit{Part 4a:} Let us now consider small $r$. We use the representation via the 4d Laplacian, i.e. we consider $U \colon \R^4 \times \T \to \R$, $U(X, t) = U(X_1, X_2, X_3, X_4, t) = \tfrac{1}{\abs{X}} u(\abs{X}, t)$. Notice that the $L^2$-norms are equivalent, i.e. 
	\begin{align*}
		\norm{\tfrac{1}{\abs{X}} f(\abs{X}, t)}_{L^2(B_R \times \T)} = \sqrt{2} \pi \norm{f}_{L^2([0, R] \times \T; r \Der(r,t))}
	\end{align*}
	holds with $B_R \subseteq \R^4$ denoting the ball of radius $R$ centered at $0$. Multiplying \eqref{eq:loc:fcr_prime_pde:pw} by $\tfrac{1}{r}$, setting $r = \abs{X}$, and recalling \eqref{eq:4dlaplace} we have 
	\begin{align}\label{eq:loc:fcr_prime_pde:singular}
		\Delta_X \fracDT{s} U = V(r) \fracDT{s+2} U - \tfrac{h(r)}{r} \fracDT{s+2} \projR[\calN \ast u^3].
	\end{align}
	By part~1 the right-hand side of \eqref{eq:loc:fcr_prime_pde:singular} lies in $L^{2}(\R^4 \times \T)$. Therefore $\fracDT{s} U \in H^2(\R^4 \times \T)$, which by Sobolev's embeddings shows $\fracDT{s} \partial_r U \in L^{\nicefrac{10}{3}}(\R^4 \times \T), \fracDT{s} U \in L^{10}(\R^4 \times \T)$. 
	Now let $R_1 > 0$ and $\psi_1 \in C_c^\infty([0, R_1))$ be a smooth cutoff function with $\psi_1 \equiv 1$ on $[0, R_2]$ for some $0 < R_2 < R_1$ and set $v_1 \coloneqq \psi_1(\abs{X}) U$.
	Then, since 
	\begin{align*}
		\Delta_X \fracDT{s} v_1 = \Delta_X \psi_1 \cdot \fracDT{s} U + \partial_r \psi_1 \cdot \partial_r \fracDT{s} U + \psi_1 \cdot \Delta_X \fracDT{s} U
	\end{align*}
	and since we may write \eqref{eq:loc:fcr_prime_pde:singular} as
	\begin{align}\label{eq:loc:fcr_prime_pde:singular2}
		\Delta_X \fracDT{s} U = V(r) \fracDT{s+2} U - r^2 h(r) \fracDT{s+2} \projR[\calN \ast U^3],
	\end{align}
	we have $\Delta_X \fracDT{s} v_1 \in L^{\nicefrac{10}{3}}(\R^4 \times \T)$ which by $L^p$-boundedness of the Riesz transform implies $\fracDT{s} v_1 \in W^{2,\nicefrac{10}{3}}(\R^4 \times \T)$. From Sobolev's embedding we have $\fracDT{s} U \in L^\infty(B_{R_2} \times \T)$, $\nabla_X \fracDT{s} U \in L^{10}(B_{R_2} \times \T)$. Repeating this argument with $0 < R_3 < R_2$ and a cutoff function $\psi_2 \in C_c^\infty([0, R_2))$, $\psi_j \equiv 1$ on $[0, R_3]$ shows $\nabla_X \fracDT{s} U \in L^\infty(B_{R_3} \times \T)$. Using \eqref{eq:loc:fcr_prime_pde:singular2}, regularity of the terms in the claim of \cref{prop:regularity:cyl:W2p} follows since 
	\begin{align*}
		\fracDT{s} \tfrac{u}{r} = \fracDT{s} U \in L^\infty,
		\qquad 
		\fracDT{s} \partial_r (\tfrac{u}{r}) = \tfrac{X}{r} \cdot \nabla_X \fracDT{s} U \in L^\infty,
		\\ r \fracDT{s} \partial_r^2 (\tfrac{u}{r}) = r \Delta_X \fracDT{s} U - 3 \tfrac{X}{r} \cdot \nabla_X \fracDT{s} U \in L^\infty
	\end{align*}
	The $L^p$-estimates follow from these since $B_{R_3} \times \T$ has finite volume.

	\textit{Part 4b:} Assume \ref{ass:regularity}, and again consider small $r$. First, $\fracDT{s} U, \nabla_X \fracDT{s} U$ are continuous by Sobolev's embedding, and continuity of $\Delta_X \fracDT{s} U$ follows from this by \eqref{eq:loc:fcr_prime_pde:singular2}. 
	Existence and continuity of higher derivatives 
	\begin{align*}
		\nabla_X^n \Delta_X \fracDT{s} U, 
		\qquad
		\nabla_X^{n+1} \fracDT{s} U, 
		\qquad
		\nabla_X^n \fracDT{s} U
	\end{align*}
	for $0 < n \leq l$ can again be shown using induction and repeatedly applying $\nabla_X$ to \eqref{eq:loc:fcr_prime_pde:singular2}. This implies continuity of all terms except the highest order one in \cref{prop:regularity:cyl:W2p} 
	since $\fracDT{s} \partial_r^n (\tfrac{u}{r}) = (\nabla_X^n \fracDT{s} U)[\tfrac{X}{r}, \dots, \tfrac{X}{r}]$. Moreover, odd $r$-derivatives of $\frac{u}{r}$ vanish at $r = 0$ since $U$ is radially symmetric. For the highest order term we have  
	\begin{align*}
		(\nabla_X^l \Delta_X \fracDT{s} U)[\tfrac{X}{r}, \dots, \tfrac{X}{r}]
		= \fracDT{s} \partial_r^{l+2}(\tfrac{u}{r}) 
		+ \fracDT{s} \partial_r^l (\tfrac{3}{r} \partial_r (\tfrac{u}{r}))
	\end{align*}
	which shows that $\fracDT{s} \partial_r^{l+2}(\tfrac{u}{r})$ is continuous away from $0$. To see the behaviour of the highest order term near $r=0$ we use the differentiability properties of $U$ and a Taylor expansion of $\fracDT{s}(\tfrac{u}{r})$ about $r=0$ as follows. Let $\fracDT{s}(\tfrac{u}{r}) = T_{l+1}(\fracDT{s} (\tfrac{u}{r}); 0) + f$ be the Taylor expansion of $\fracDT{s}(\frac{u}{r})$ of degree $l+1$ about $r=0$ with remainder $f$. Then we have
	\begin{align*}
		\fracDT{s} \partial_r^l (\tfrac{3}{r} \partial_r (\tfrac{u}{r}))
		= \partial_r^l \left(\tfrac{3}{r} \partial_r \left[T_{l+1}(\fracDT{s} \tfrac{u}{r}; 0)\right]\right)
		+ \partial_r^l \left(\tfrac{3}{r} \partial_r f\right)
		= \frac{3 (-1)^l}{l!}\frac{\fracDT{s}(\tfrac{u}{r})_r(0)}{r^{l+1}} + \landauo(\tfrac{1}{r})
		= \landauo(\tfrac{1}{r})
	\end{align*}
	as $r \to 0$ since $\fracDT{s}(\tfrac{u}{r})_r(0)=0$ by radial symmetry. This shows that $r \fracDT{s} \partial_r^{l+2}(\tfrac{u}{r}) \to 0$ as $r \to 0$.
\end{proof}

Next, in \cref{prop:regularity:cyl:w}, similar to \cref{prop:regularity:w} we discuss the second nonlinearity \eqref{eq:material:nonii}.

\begin{proposition}\label{prop:regularity:cyl:w}
	Let $u \in \fcrd$ be a critical point of $\fcr$ and let the nonlinearity be given by $N(w) = \calN \ast w^3$. Define $w = w_1 + w_2$ where
	\begin{align*}
		w_1 = \projR[w] = (\calN\ast)^{-1} u,
		\quad
		w_2 = \projS[w] = (-\partial_r^2 - \tfrac{1}{r} \partial_r + \tfrac{1}{r^2} - V(x) \partial_t^2)^{-1} (h(x) \partial_t^2 \projS[u^3])
	\end{align*}
	Then the functions 
	\begin{align*}
		\max\set{r, 1} \fracDT{s} \tfrac{w}{r},
		\qquad
		\max\set{r, 1} \fracDT{s} \partial_r (\tfrac{w}{r}),
		\qquad 
		r \fracDT{s} \partial_r^2 (\tfrac{w}{r})
	\end{align*}
	lie in $L^{p}([0, \infty) \times \T; r \Der(r,t))$ for all $s \in \R$ and $p \in [2, \infty]$, and $w$ solves
	\begin{align*}
		(-\partial_r^2 - \tfrac1r \partial_r + \tfrac{1}{r^2} - V(x) \partial_t^2) w + h(x) \partial_t^2 (\calN \ast w)^3 = 0.
	\end{align*}
	If \ref{ass:regularity} holds, then the terms 
	\begin{align*}
		\max\set{r, 1} \fracDT{s} \partial_r^n (\tfrac{w}{r}) 
		\quad\text{for}\quad 0 \leq n \leq l+1
		\qquad \text{as well as} \qquad 	
		r \fracDT{s} \partial_r^{l+2} (\tfrac{w}{r})
	\end{align*}
	lie in $L^2([0, \infty) \times \T; r \Der (r,t)) \cap C_b([0, \infty) \times \T)$. 
	Moreover, the second term vanishes at $r = 0$, and the same holds for the first term when $n$ is odd. 
\end{proposition}
\begin{proof}
	We follow \cref{prop:regularity:w}, and define $w_2$ by
	\begin{align*}
		w_2 \in \fcrd_2 \coloneqq
		\set{v \in H^1([0, \infty) \times \T; r \Der(r,t)) \colon \tfrac{1}{r} v \in L^2([0, \infty) \times \T; r \Der(r,t)), \hat v_k \equiv 0 \text{ for } k \in \regularK \cup \set{0}}
	\end{align*}
	and 
	\begin{align*}
		\int_{[0,\infty) \times \T} \left( \partial_r w_2 \cdot \partial_r v + \tfrac{1}{r^2} w_2 \cdot v + V(r) \partial_t w_2 \cdot \partial_t v \right) \der[r]{(r, t)}
		= - \int_{[0, \infty) \times \T} \left(h(r) \partial_t^2 \projS[u^3] \cdot v\right) \der[r]{(r,t)}
	\end{align*}
	for all $v \in \fcrd_2$.
	Regularity of $w_1$ follows from \cref{prop:regularity:cyl:W2p}, and the arguments therein can also be used to show regularity of $w_2$.
\end{proof}

As the last part of this chapter, we discuss regularity of the electromagnetic wave profiles. 

\begin{proof}[Proof of \cref{thm:main} for cylindrical geometries]
	\textit{Part 1:} Let $u \in \fcrd$ be a nontrivial critical point of $\fcr$. For the nonlinearity $N(w) = \calN \ast w^3$ set $w \coloneqq u$, else let $w$ be from \cref{prop:regularity:cyl:w}. Define $W \coloneqq \partial_t^{-1} w$ and the electromagnetic fields by
	\begin{align*}
		\bfD(\bfx, t) &= \epsilon_0 \left( w + \calG \ast w + h(r) N(w)\right) \cdot \begin{pmatrix}
			-\nicefrac yr \\ \nicefrac xr \\ 0
		\end{pmatrix},
		& \bfE(\bfx, t) &= w(r, t - \tfrac{1}{c} z) \cdot \begin{pmatrix}
			-\nicefrac yr \\ \nicefrac xr \\ 0
		\end{pmatrix},
		\\ \bfB(\bfx, t) &= - \tfrac{1}{c} w \cdot \begin{pmatrix}
			\nicefrac xr \\ \nicefrac yr \\ 0
		\end{pmatrix} - (\tfrac{1}{r} W + W_r) \cdot \begin{pmatrix}
			0 \\ 0 \\ 1
		\end{pmatrix},
		& \bfH(\bfx, t) &= \tfrac{1}{\mu_0} \bfB(\bfx, t)
	\end{align*}
	By a straightforward calculation one sees that $\bfE, \bfD, \bfB, \bfH$ solve Maxwell's equations \eqref{eq:maxwell}, \eqref{eq:material}, so it remains to show their regularity. For simplicity we only consider $\bfE$ and only discuss spatial derivatives. 
	Abbreviating $p(\bfx) \coloneqq (-y, x, 0)$, denoting the Euclidean scalar product in $\R^3$ by $\ip{\impvar}{\impvar}$ and the space derivative by $D_\bfx$, we have
	\begin{align*}
		\bfE 
		&= \frac{w}{r} p 
		\\ &= r \frac{w}{r} \frac{p}{r}, 
		\\ D_\bfx \bfE[h] 
		&= \frac{w}{r} D_\bfx p[h] + \tfrac{1}{r} \partial_r \left(\frac{w}{r}\right) \Ip{\bfx}{h} p
		\\ &= \frac{w}{r} D_\bfx p[h] + r \partial_r \left(\frac{w}{r}\right) \Ip{\frac{\bfx}{r}}{h} \frac{p}{r}
		\\ D_\bfx^2 \bfE[h_1, h_2] 
		&= \tfrac{1}{r} \partial_r \left(\frac{w}{r}\right) \left[\Ip{\bfx}{h_1} D_\bfx p[h_2] + \Ip{\bfx}{h_2} D_\bfx p[h_1] + \Ip{h_1}{h_2} p\right] 
		\\ &\qquad+ \left(\tfrac{1}{r}\partial_r\right)^2\left(\frac{w}{r}\right) \Ip{\bfx}{h_1} \Ip{\bfx}{h_2} p,
		\\ &= \partial_r \left(\frac{w}{r}\right) \left[\Ip{\frac{\bfx}{r}}{h_1} D_\bfx p[h_2] + \Ip{\frac{\bfx}{r}}{h_2} D_\bfx p[h_1] + \Ip{h_1}{h_2}\frac{p}{r}\right] 
		\\ &\qquad+ \left[ r \partial_r^2 \left(\frac{w}{r}\right) - \partial_r \left(\frac{w}{r}\right) \right] \Ip{\frac{\bfx}{r}}{h_1} \Ip{\frac{\bfx}{r}}{h_2} \frac{p}{r},
	\end{align*}
	so \cref{prop:regularity:cyl:W2p,prop:regularity:cyl:w} show that these terms lie in $L^p$ for $p \in [2, \infty]$.

	\textit{Part 2:} Let us now assume \ref{ass:regularity}. We need to show that higher order derivatives exist, are continuous and square-integrable. Away from $r = 0$, this is clear by \cref{prop:regularity:cyl:W2p,prop:regularity:cyl:w}, so it remains to show continuity of derivatives in $r=0$. First, by induction one can show that for $0 \leq n \leq l+2$ the derivative $D_\bfx^n \bfE$ can be written as a sum
	\begin{align*}
		D_\bfx^n \bfE = 
		\sum_{j=\ceil{\frac{n-1}{2}}}^{n}(\tfrac{1}{r}\partial_r)^{j} \left(\frac{w}{r}\right) \cdot p_{n,j},
	\end{align*}
	where $p_{n,j}(\bfx) \colon (\R^3)^n \to \R^3$ is symmetric, $n$-multilinear, and its coefficients are homogeneous polynomials of degree $2 j + 1 - n$ in $\bfx$. We use Taylor approximation and write $\tfrac{w}{r} = T_{n-1}(\tfrac{w}{r}; 0) + f$ with Taylor polynomial $T_{n-1}(\tfrac{w}{r}; 0)$ and remainder $f$. 
	
	Let us next consider summands with $j < n$. Recall that all odd Taylor coefficients are zero, so $q_{n,j} \coloneqq (\tfrac{1}{r} \partial_r)^{j} T_{n-1}(\tfrac{w}{r}; 0)$ is an even polynomial. In addition, we can estimate the remainder via $(\tfrac{1}{r} \partial_r)^{j} f = \landauo(r^{n - 1 - 2 j})$ as $r \to 0$. Thus
	\begin{align*}
		(\tfrac{1}{r}\partial_r)^j \left(\frac{w}{r}\right) \cdot p_{n,j}
		= (q_{n,j}(r) + \landauo(r^{n-1-2j})) p_{n,j} \to q_{n,j}(0) p_{n,j}(0)
	\end{align*}
	as $r \to 0$. Now let $j = n$. Similar to the above arguments, one can show 
	\begin{align*}
		\left[(\tfrac{1}{r}\partial_r)^{n} - \frac{1}{r^n} \partial_r^n\right] \left(\frac{w}{r}\right) 
		= q_{n,n}(r) + \landauo(r^{-n-1})
	\end{align*}
	as $r \to 0$ for some polynomial $q_{n,n}$. Thus
	\begin{align*}
		D_\bfx^n \bfE
		= \frac{1}{r^n} \partial_r^n (\tfrac{w}{r}) \cdot p_{n, n} + \sum_{j=\ceil{\frac{n-1}{2}}}^{n} q_{n,j}(r) p_{n,j}(\bfx) + \landauo(1)
		\to \sum_{j=\ceil{\frac{n-1}{2}}}^{n} q_{n,j}(0) p_{n,j}(0)
	\end{align*}
	as $r \to 0$ by \cref{prop:regularity:cyl:W2p,prop:regularity:cyl:w}. Since the argument for existence of infinitely many solutions is the same as for slab geometries at the end of \cref{sec:regularity}, this completes the proof. Observe that $p_{n,j}(0) = 0$ for $j \neq \frac{n+1}{2}$, so in particular all even derivatives of $\bfE$ vanish at $0$.
\end{proof}

	\section*{Acknowledgement} 
	
	Funded by the Deutsche Forschungsgemeinschaft (DFG, German Research Foundation) – Project-ID 258734477 – SFB 1173

	\printbibliography
\end{document}